\documentclass[11pt,twoside,reqno]{amsart}
\usepackage[all]{xy} 
\usepackage{enumitem}  
\usepackage {amssymb,latexsym,amsthm,amsmath}
\usepackage[colorlinks, linkcolor=blue, citecolor=red]{hyperref}
\usepackage{longtable}
\usepackage{graphicx}
\usepackage{mathtools}
\usepackage{listings}

\usepackage{caption}

\usepackage{xcolor}
\usepackage[utf8]{inputenc}
\usepackage{tikz,quiver}

\long\def\symbolfootnote[#1]#2{\begingroup%
\def\thefootnote{\fnsymbol{footnote}}\footnote[#1]{#2}\endgroup}

\newcommand{\tra}{\ensuremath{{}^t}}

\newcommand{\F}{\mathbb F}

\newcommand{\ad}{\mathrm{ad}}

\newcommand{\Char}{\textup{char}}

\newcommand{\g}{\mathfrak{g}}
\newcommand{\h}{\mathfrak{h}}

\newcommand{\diag}{\textup{diag}}

\makeatletter
\def\imod#1{\allowbreak\mkern10mu({\operator@font mod}\,\,#1)}
\makeatother

\newtheorem{theorem}{Theorem}[section]
\newtheorem{lemma}[theorem]{Lemma}

\newtheorem{proposition}[theorem]{Proposition}
\newtheorem*{theorem*}{Theorem}
\newtheorem*{theorema}{Theorem A}
\theoremstyle{definition}
\newtheorem{remark}[theorem]{Remark}

\newtheorem{definition}[theorem]{Definition}

\numberwithin{equation}{section}

\newcommand{\ignore}[1]{}

\newcommand{\mynote}[1]{}

\def \Sp {\mathrm{Sp}}

\def \G {\mathfrak{G}}

\newcommand{\Fr}{\mathrm{Fr}}
\newcommand{\Tr}{\operatorname{Tr}}

\begin{document}

\setcounter{section}{0}
\setcounter{tocdepth}{1}
\title[$2$-generation of Chevalley Lie algebras over finite fields]{$2$-generation of Chevalley Lie algebras over finite fields}
\author{Ram Karan Choudhary}
\email{ramkchoudhary1997@gmail.com}

\author{Harish Kishnani}
\email{harishkishnani11@gmail.com}
\author{Anupam Singh}
\email{anupamk18@gmail.com}
\address{Indian Institute of Science Education and Research Pune, Dr Homi Bhabha Road, Pashan, Pune 411008 India}

\thanks{The first-named author acknowledges the support of IISER Pune for the institute postdoctoral fellowship. The second-named author acknowledges the support of ANRF-NPDF grant PDF/2025/002961 for the postdoctoral fellowship. The third-named author is funded by an ANRF-MATRICS Grant ANRF/ARGM/2025/000095/MTR}
\date{\today}

\begin{abstract}
In this article, we prove that every split simple Chevalley Lie algebra over a finite field of characteristic $p\neq 2$ of classical types $B_n$, $C_n$, and $D_n$, and exceptional types $G_2$, $F_4$, $E_6$, $E_7$, and $E_8$, is $2$-generated. Combined with the result of Cantor, Jezernik, and Zozaya for the classical type $A_n$ [\emph{Two-generation of traceless matrices over finite fields}, Linear Algebra Appl., vol. 713, pages 1-17 (2025)], this establishes $2$-generation for all split simple Chevalley Lie algebras over finite fields of characteristic $p\neq 2$.
\end{abstract}
\maketitle


\section{Introduction}
The problem of determining the minimum number of generators for an algebraic structure is a fundamental question across various domains of algebra. It is closely related to the structure of the object, the nature of its sub-objects, and the extent to which its elements can be combined to recover the entire structure. In particular, the question of whether a given algebraic structure can be generated by two elements has attracted considerable attention in group theory, ring theory, and Lie theory. Two-generation results are significant not only because they provide economical generating sets, but also because they often reveal structural properties and lead to explicit descriptions of the objects under consideration.

A fundamental result in finite group theory asserts that every finite non-abelian simple group can be generated by two elements. The development of this result spans several decades. Miller initiated the study of two-generation for alternating groups in \cite{Miller01}, and Steinberg established the corresponding result for finite simple groups of Lie type in \cite{Steinberg62}. Aschbacher and Guralnick settled the remaining case of sporadic simple groups in \cite{AsGu84}, thereby completing the proof of the two-generation theorem for finite non-abelian simple groups. This result has inspired further investigations into the generation of finite simple groups under additional constraints on the generating elements. For instance, Burness and Fusari \cite{BurFus25} proved that every finite simple transitive group can be generated by two conjugate derangements. 
Another important question, raised by Steinberg in \cite{Steinberg62}, asks whether every finite simple group satisfies the stronger property of being $\frac{3}{2}$-generated. A group $G$ is called $\frac{3}{2}$-generated if, for every non-identity element $x\in G$, there exists an element $y\in G$ such that $G=\langle x,y\rangle$. Guralnick and Kantor \cite{GuKa00} answered this question affirmatively. Further related investigations can be found in \cite{BrGuKa08, GKKL11, BuGu13, Ha17}. For a broader overview of generation problems for finite simple groups and their subsequent developments, we refer the reader to a survey article by Burness \cite{Burness19}.

Generation by a small number of elements is also an important problem in matrix algebras (see \cite{AAR96, neumannpraeger, petrenkosidki, Ka10}). Analogous generation problems have also been studied extensively in the context of Lie algebras. Kuranishi \cite{Kuranishi} established that simple Lie algebras over fields of characteristic zero are $2$-generated, and explicit generating pairs were subsequently constructed by Grozman and Leites \cite{GrLe96}. Stronger results were obtained by Ionescu \cite{Ionescu}, who established $\frac{3}{2}$-generation over fields of characteristic zero, and by Bois \cite{Bois}, who extended this property to infinite fields of characteristic $p>3$. In a related direction, Chistopolskaya investigated pairs of nilpotent generators for the Lie algebra of trace-zero matrices over infinite fields of characteristic different from $2$, and for the symplectic Lie algebra over algebraically closed fields of characteristic zero (see \cite{Chistopolskaya_sln, Chistopolskaya_spn}).

Let $\mathfrak g$ be a finite-dimensional simple Lie algebra over $\mathbb {C}$. We know that for such an algebra, there exists a Chevalley basis that defines the Lie algebra over $\mathbb Z$, denoted as $\mathfrak g_{\mathbb Z}$. Let $F$ be a field. One defines a Lie algebra $\mathfrak g_F = \mathfrak g_{\mathbb Z}\otimes F$, known as the Chevalley Lie algebra. For a detailed account, we refer the reader to the excellent books~\cite{Ca89, Se67}. The developments over infinite fields motivate the study of two-generation for split simple Chevalley Lie algebras over finite fields, which is the focus of the present work. The recent paper of Cantor, Jezernik, and Zozaya \cite{cjz} makes an
important step in this direction. They prove that the Lie algebra $\mathfrak{psl}_n(\mathbb F_q)$, is $2$-generated except for the pairs $(n, \Char(\F_q)) = (3,3)$ and $(4,2)$. In particular, the $A_n$-type case is settled in
\cite{cjz}. The purpose of the present article is to extend the $A_n$-type result of \cite{cjz} to all split simple Chevalley Lie algebras over finite fields. Our main result is the following.


\begin{theorema}\label{thm:main result}
Let $\mathbb F_q$ be a finite field of characteristic $p\neq 2$, and let $\mathfrak g_{\mathbb F_q}$ be a split simple Chevalley Lie algebra over $\mathbb F_q$. Then $\mathfrak g_{\mathbb F_q}$ is generated by two elements.
\end{theorema}
\noindent We prove this theorem in Section~\ref{section:proof} as a consequence of Theorems~\ref{thm:two-generation classical} and~\ref{thm:two-generation exceptional}. 

{\bf Approach to the proof:} The proof for the classical case is explicit. For the proof of each classical type of Lie algebra $\mathfrak{g}$, we first construct a suitable matrix $\widetilde{C}_{B_n} \in \mathfrak{so}_{2n+1}(\mathbb{F}_q)$, $\widetilde{C}_{C_n} \in \mathfrak{sp}_{2n}(\mathbb{F}_q)$, and $\widetilde{C}_{D_n} \in \mathfrak{so}_{2n}(\mathbb{F}_q)$, each obtained as a block matrix involving the companion matrix of a monic irreducible polynomial in $\mathbb{F}_q[x]$ of degree $n$ (see Section~\ref{Section: prelim}). We then construct an element $X \in \mathfrak{g}$ having a nonzero component in every root space $\mathfrak{g}_{\alpha}$ of the corresponding classical Lie algebra. The generating elements are obtained by appropriately tweaking these two elements. We show that there is a matrix $P$, different for each type (described in Section~\ref{Section: Rationality of matrices}), so that $ P\widetilde{C}_{B_n}P^{-1}, P\widetilde{C}_{C_n}P^{-1}$, and $P\widetilde{C}_{D_n}P^{-1}$ are diagonal matrices; further, these are consistent elements of the respective Cartan subalgebras $\mathfrak{h}$ over $\mathbb{F}$, where $\mathbb{F}$ is the splitting field of the chosen irreducible polynomial over $\mathbb{F}_q$, of degree $n$. Recall that an element $T \in \mathfrak{h}$ is called \emph{consistent} if the root evaluations $\alpha(T)$, for $\alpha \in \Phi$, are pairwise distinct and nonzero. For each Lie algebra type, we show that these diagonal matrices, together with the corresponding element $X$ defined above over $\mathbb{F}$, generate the corresponding Lie algebra over $\mathbb{F}$. To obtain generating pairs over the base field $\mathbb{F}_q$, we employ a descent technique by conjugating the generators over $\mathbb{F}$ with $P$ (see Lemma~\ref{lem:rational-Cn} and Propositions~\ref{prop:two-gen-odd} and~\ref{prop:two-gen-even}). 

The proof of the exceptional case also uses GAP computations. For each exceptional type Lie algebra, we use a result of Alon~\cite{alon} to establish the existence of a consistent element in the Cartan subalgebra which, together with an element $X$ constructed analogously to the one described above, generates the Lie algebra, except for certain small values of $q$. For these remaining values, we explicitly construct generating pairs and use \textsc{GAP}~\cite{Gap} to verify that they generate the corresponding Lie algebras. Consequently, our construction establishes the two-generation property and provides an explicit procedure for computing a generating pair for every split simple Chevalley Lie algebra over $\mathbb{F}_q$.

A small number of generators is required in computational topics and Waring-like questions (see~\cite{KiSi26}). We hope this work finds its application in these subjects. The paper is organized as follows. In Section \ref{Section: prelim}, we recall some basic notions and properties, and establish several preliminary results. Section \ref{Section: Rationality of matrices} is devoted to technical computations concerning classical Lie algebras over finite fields. In Section \ref{section:classical}, we handle the case of classical split simple Chevalley Lie algebras, while Section \ref{section:exceptional} deals with exceptional split simple Chevalley Lie algebras. The proof of the main theorem is presented in Section \ref{section:proof}. Finally, Section \ref{Section: GAP Codes} contains the details of the \textsc{GAP} codes used in the paper. The corresponding GAP code is also available in the GitHub repository \cite{CKS26}. It consists of two scripts: `G2-2-Generation.g', which verifies the 2-generation of the split simple Chevalley Lie algebra $\mathfrak{G}_2$ over $\mathbb{F}_q$, for $q=5,7,$ and $11$; and `F4-Positive-Roots.g', which computes the positive roots, distinct linear factors, and other related data associated with the root system of type $\mathfrak{F}_4$. Analogous computations can be carried out for the other exceptional Lie algebras with necessary modifications to these \textsc{GAP} codes. 
\section{Preliminaries and the main approach}\label{Section: prelim}
In this section, we briefly recall some basic notions and properties of Lie algebras needed for our work. We also establish several preliminary results that will be used throughout the article and play an essential role in proving the main results in the subsequent sections. Throughout, the split simple Chevalley Lie algebras over a field $\mathbb{F}$ of classical types $A_n$, $B_n$, $C_n$, and $D_n$ are used in matrix form $\mathfrak{sl}_{n+1}(\mathbb{F})$, $\mathfrak{so}_{2n+1}(\mathbb{F})$, $\mathfrak{sp}_{2n}(\mathbb{F})$, and $\mathfrak{so}_{2n}(\mathbb{F})$, respectively. Similarly, those of exceptional types $G_2$, $F_4$, $E_6$, $E_7$, and $E_8$ are denoted by $\mathfrak{G}_2(\mathbb{F})$, $\mathfrak{F}_4(\mathbb{F})$, $\mathfrak{E}_6(\mathbb{F})$, $\mathfrak{E}_7(\mathbb{F})$, and $\mathfrak{E}_8(\mathbb{F})$, respectively. In this article, we work over a finite field $\F_q$ of characteristic $p\neq 2$. 
We realize the classical types as the matrix Lie algebras $\mathfrak{so}_{2n+1}(\F_q), \mathfrak{sp}_{2n}(\F_q)$, and $\mathfrak{so}_{2n}(\F_q)$ with the standard nondegenerate bilinear forms
\[
S_{B_n}= \begin{pmatrix} 0 & 0 & I_n\\ 0 & 1 & 0\\ I_n & 0 & 0
\end{pmatrix},
\qquad
S_{C_n}= \begin{pmatrix} 0 & I_n\\ -I_n & 0 \end{pmatrix},
\quad \text{and} \quad
S_{D_n}= \begin{pmatrix} 0 & I_n\\ I_n & 0 \end{pmatrix}
\] 
defining them. Since $p\neq 2$, these matrix realizations coincide with the corresponding Chevalley Lie algebras obtained by base change from a Chevalley basis over $\mathbb Z$ (see~\cite{Ca89, chevalley}). Over a finite field, unlike over $\mathbb C$, these need not be simple but are simple modulo their center. Except for the $A_n$ type; when $p\neq 2$ and $n\geq 1$, these Lie algebras are simple, which is what we are dealing with in this article.

\subsection{Cartan subalgebra and roots of classical types}
Since the question we are dealing with is already solved for the $A_n$ type, we mostly focus on other classical types here in this subsection. We use the following two kinds of notation for the elementary matrices.

\begin{itemize}
\item $E_{ij}$ denotes the full $N\times N$ elementary matrix, where $N$ is the size of the full matrix ($N=2n+1$ for type $B_n$, $N=2n$ for types $C_n$ and $D_n$). The indices for $B_n$ range over $\{1, 2, \dots, n, 0, -1, -2, \dots, -n\}$, and $\{1, 2, \dots, n, -1, -2, \dots, -n\}$ for $C_n$ and $D_n$. This $E_{ij}$ is used in the root-vector table and in the sums defining a particular generating element $X$ later.
\item $F_{ij}$ denotes the $n\times n$ elementary matrix, with indices $i, j \in\{1, \dots, n\}$. It is used only inside the block forms of the particular element $X$.
\end{itemize}
Let the Lie algebra $\g$ be one of $\mathfrak{so}_{2n+1}, \mathfrak{sp}_{2n}$ and $\mathfrak{so}_{2n}$ over a field $\F$ of characteristic $p\neq 2$, realized as a subalgebra of $M_N(\F)$ preserving a bilinear form with matrix $S$ (defined above for each type):
\[
\g= \{Z\in M_N(\F) : Z^{\tra}S + SZ = 0\}.
\]
The Cartan subalgebra $\h$ consists of the diagonal matrices
$\h = \left\{H(a) = \displaystyle\sum_{i=1}^n a_i(E_{i,i}-E_{-i,-i}) : a_i\in \F \right\}$ with the convention $a_0 =0$ for type $B_n$. Define $\varepsilon_i\in\h^*$ by
$\varepsilon_i(H(a))=a_i$. Then the root system is given by
\[
\Phi=
\begin{cases}
\{\pm\varepsilon_i,\ \pm\varepsilon_i\pm\varepsilon_j\ (i<j)\}, & \text{for } B_n,\\[1mm]
\{\pm 2\varepsilon_i,\ \pm\varepsilon_i\pm\varepsilon_j\ (i<j)\}, & \text{for } C_n,\\[1mm]
\{\pm\varepsilon_i\pm\varepsilon_j\ (i<j)\}, & \text{for } D_n.
\end{cases}
\]
The Lie algebra decomposes as $\g = \h \oplus \bigoplus_{\alpha \in \Phi}\g_\alpha$ where each root space $\g_\alpha$ is one-dimensional (see~\cite{Humphreys}). A root vector for $\alpha$ is a nonzero element $e_\alpha \in \g_\alpha$. 

\begin{definition}\label{consistent}
An element $T \in \mathfrak{h}$ is called \emph{consistent} if the root evaluations $\alpha(T)$, for $\alpha \in \Phi$, are pairwise distinct and non-zero.
\end{definition}
\noindent The notion of a consistent element was implicitly employed by Kuranishi~\cite{Kuranishi} in his study of the two-generation property of semisimple Lie algebras over fields of characteristic zero. 
In~\cite{cjz}, the authors established the two-generation property of $A_n$ by exhibiting a consistent element $T \in \mathfrak{h}$ together with an element $X \in \mathfrak{g}=A_n$ having a non-zero component in every root space $\mathfrak{g}_{\alpha}$. In the following subsection, we explicitly describe the root vectors for the remaining classical types, which will be needed in our subsequent arguments.

\subsection{Root vectors of classical types in the standard matrix realization}
Using the commutator identity $[E_{ij}, E_{kl}] = \delta_{jk}E_{il}-\delta_{li}E_{kj}$, one checks that the following matrices are root vectors (see~\cite[Chapter~12]{erdmann}). For type $B_n$ the extra coordinate is $0$; for all types, $-i$ denotes the coordinate paired with $i$ by the form $S$ defined earlier for each type.

\noindent\textbf{Type $B_n$} (coordinates $1, \dots, n, 0, -1, \dots, -n$):
\[
\begin{array}{c|c}
\text{Root }\alpha & \text{Root vector }e_\alpha\\ \hline
\varepsilon_i-\varepsilon_j\ (i\ne j) & E_{i,j}-E_{-j,-i}\\
\varepsilon_i+\varepsilon_j\ (i<j) & E_{i,-j}-E_{j,-i}\\
-\varepsilon_i-\varepsilon_j\ (i<j) & E_{-i,j}-E_{-j,i}\\
\varepsilon_i & E_{i,0}-E_{0,-i}\\
-\varepsilon_i & E_{-i,0}-E_{0,i}\\ \hline
\end{array}
\]

\noindent\textbf{Type $C_n$} (coordinates $1, \dots, n, -1,\dots, -n$):
\[
\begin{array}{c|c}
\text{Root }\alpha & \text{Root vector }e_\alpha\\ \hline
\varepsilon_i-\varepsilon_j\ (i\ne j) & E_{i,j}-E_{-j,-i}\\
\varepsilon_i+\varepsilon_j\ (i<j) & E_{i,-j}+E_{j,-i}\\
-\varepsilon_i-\varepsilon_j\ (i<j) & E_{-i,j}+E_{-j,i}\\
2\varepsilon_i & E_{i,-i}\\
-2\varepsilon_i & E_{-i,i} \\\hline
\end{array}
\]
\noindent\textbf{Type $D_n$} (coordinates $1, \dots, n, -1,\dots, -n$):
\[
\begin{array}{c|c}
\text{Root }\alpha & \text{Root vector }e_\alpha\\ \hline
\varepsilon_i-\varepsilon_j\ (i\ne j) & E_{i,j}-E_{-j,-i}\\
\varepsilon_i+\varepsilon_j\ (i<j) & E_{i,-j}-E_{j,-i}\\
-\varepsilon_i-\varepsilon_j\ (i<j) & E_{-i,j}-E_{-j,i}
\end{array}
\]

\noindent Note that each listed matrix above is nonzero and satisfies $e_\alpha^{\tra}S + Se_\alpha = 0$, so it lies in $\g$. A direct computation gives $[H,e_\alpha] = \alpha(H)e_\alpha$ for all $H\in\h$, so $e_\alpha\in\g_\alpha$.


\subsection{Companion block matrices}\label{subsec:companion}

In this subsection, we construct a specific matrix in each classical Lie algebra over $\mathbb F_q$ that will play a role in the construction of the generating matrices. Let $\F_q$ be a finite field of characteristic $p\neq 2$, and let $f(x)=x^n+a_{n-1}x^{n-1}+\cdots+a_1x+a_0\in\F_q[x] $ be a monic polynomial of degree $n$. Let $C_f\in M_n(\F_q)$ denote its companion matrix,
\[
C_f=
\begin{pmatrix}
0 & 1 & 0 & \cdots & 0\\
0 & 0 & 1 & \cdots & 0\\
\vdots & \vdots & \vdots & \ddots & \vdots\\
0 & 0 & 0 & \cdots & 1\\
-a_0 & -a_1 & -a_2 & \cdots & -a_{n-1}
\end{pmatrix}.
\]
With this convention, the Vandermonde column vector $v_i^{\tra} = (1, t_i, t_i^2, \dots, t_i^{n-1})^{\tra}$ 
is a right eigenvector of $C_f$ with eigenvalue $t_i$, i.e., 
$C_f v_i^{\tra} = t_i v_i^{\tra}$. Indeed, the first $n-1$ coordinates of $C_fv_i^{\tra}$ are $t_i, t_i^2,\dots, t_i^{n-1}$, and the last coordinate is $-a_0 - a_1t_i -\cdots - a_{n-1}t_i^{n-1} =t_i^n$ where we use $f(t_i)=0$. Hence, we have $C_fv_i^{\tra} = t_i(1, t_i,\dots, t_i^{n-1})^{\tra} = t_iv_i^{\tra}$. Thus, transposing this equation gives
\[
v_iC_f^{\tra}=t_iv_i,
\]
so the Vandermonde row vector $v_i = (1, t_i,\dots, t_i^{n-1})$ is a left eigenvector of $C_f^{\tra}$ with eigenvalue $t_i$.

Now, we define the following matrices: 
\[\tilde{C}_{B_n} = \begin{pmatrix} C_f & 0 & 0\\ 0 & 0 & 0\\
0 & 0 & -C_f^{\tra} \end{pmatrix} \in M_{2n+1}(\F_q),
\ \ \tilde{C}_{C_n} = \begin{pmatrix} C_f & 0\\ 0 & -C_f^{\tra} \end{pmatrix} \in M_{2n}(\F_q),
\ \  \text{and}\] 
\[\tilde{C}_{D_n} = \begin{pmatrix} C_f & 0\\ 0 & -C_f^{\tra}
\end{pmatrix} \in M_{2n}(\F_q). \]
Then we have the following,
\begin{lemma}[Companion block matrices in classical Lie algebras]\label{lem:companion-blocks}
Let $p\neq 2$ and let $f(x) \in\F_q[x]$ be monic of degree $n$ with companion matrix $C_f$ as above. Then, $\tilde{C}_{B_n} \in \mathfrak{so}_{2n+1}(\F_q), \tilde{C}_{C_n} \in \mathfrak{sp}_{2n}(\F_q)$ and $\tilde{C}_{D_n} \in \mathfrak{so}_{2n}(\F_q)$. Moreover, over a splitting field of $f(x)$, the eigenvalues of $\tilde C$ are
\[
\begin{cases}
0,\ \lambda_1, \dots, \lambda_n,\ -\lambda_1,\dots,-\lambda_n, & \text{for } B_n, \\
\lambda_1, \dots, \lambda_n,\ -\lambda_1,\dots,-\lambda_n, & \text{for } C_n,\\ 
\lambda_1, \dots, \lambda_n,\ -\lambda_1,\dots,-\lambda_n, & \text{for } D_n,
\end{cases}
\]
where $\lambda_1, \dots, \lambda_n$ are the roots of $f(x)$.
\end{lemma}
\begin{proof}
To prove this, we verify the condition defining each type.

\noindent\textbf{Type $B_n$.}
The form is $S_{B_n}= \begin{pmatrix} 0 & 0 & I_n\\ 0 & 1 & 0\\ I_n & 0 & 0 \end{pmatrix}$ and $\mathfrak{so}_{2n+1}(\F_q) = \{X\in M_{2n+1}(\F_q):X^{\tra}S_{B_n}+S_{B_n}X=0\}$. With blocks of sizes $n,1,n$,
\[S_{B_n}\tilde{C}_{B_n} = \begin{pmatrix} 0 & 0 & I_n\\
0 & 1 & 0\\ I_n & 0 & 0 \end{pmatrix} \begin{pmatrix} C_f & 0 & 0\\ 0 & 0 & 0\\ 0 & 0 & -C_f^{\tra}
\end{pmatrix} = \begin{pmatrix} 0 & 0 & -C_f^{\tra}\\
0 & 0 & 0\\ C_f & 0 &  0 \end{pmatrix}.
\]
Taking transposes,
\[
(S_{B_n}\tilde{C}_{B_n})^{\tra}
=
\begin{pmatrix}
0 & 0 & C_f^{\tra}\\
0 & 0 & 0\\
-C_f & 0 & 0
\end{pmatrix}
=
-S_{B_n}\tilde{C}_{B_n}.
\]
Since $S_{B_n}$ is symmetric, $\tilde{C}_{B_n}^{\tra}S_{B_n} + S_{B_n}\tilde{C}_{B_n} = 0$, therefore, we have $\tilde{C}_{B_n}\in\mathfrak{so}_{2n+1}(\F_q)$.

\noindent\textbf{Type $C_n$.}
The form is $S_{C_n}= \begin{pmatrix} 0 & I_n\\ -I_n & 0
\end{pmatrix}$ and $\mathfrak{sp}_{2n}(\F_q) =
\{X\in M_{2n}(\F_q) : X^{\tra}S_{C_n}+S_{C_n}X=0\}$. Compute
\[\tilde{C}_{C_n}^{\tra}S_{C_n}
= \begin{pmatrix} C_f^{\tra} & 0\\ 0 & -C_f \end{pmatrix} \begin{pmatrix} 0 & I_n\\ -I_n & 0 \end{pmatrix}
= \begin{pmatrix} 0 & C_f^{\tra}\\ C_f & 0 \end{pmatrix},
\] 
and
\[ S_{C_n}\tilde{C}_{C_n} = \begin{pmatrix}
0 & I_n\\ -I_n & 0 \end{pmatrix}
\begin{pmatrix} C_f & 0\\ 0 & -C_f^{\tra} \end{pmatrix}
= \begin{pmatrix} 0 & -C_f^{\tra} \\ -C_f & 0 \end{pmatrix}.
\]
Hence, we get $\tilde{C}_{C_n}^{\tra}S_{C_n}+S_{C_n}\tilde{C}_{C_n}=0$, therefore, we have $\tilde{C} _{C_n} \in \mathfrak{sp}_{2n}(\F_q)$.

\noindent\textbf{Type $D_n$.}
The form is $S_{D_n}= \begin{pmatrix} 0 & I_n\\ I_n & 0
\end{pmatrix}$ and $\mathfrak{so}_{2n}(\F_q) = \{X\in M_{2n}(\F_q):X^{\tra}S_{D_n}+S_{D_n}X=0\}$.
Compute
\[ S_{D_n}\tilde{C}_{D_n} = \begin{pmatrix}
0 & I_n\\ I_n & 0 \end{pmatrix} \begin{pmatrix} C_f & 0\\
0 & -C_f^{\tra} \end{pmatrix}
= \begin{pmatrix} 0 & -C_f^{\tra}\\ C_f & 0 \end{pmatrix}.
\]
Taking transposes,
\[(S_{D_n}\tilde{C}_{D_n})^{\tra} = \begin{pmatrix} 0 & C_f^{\tra}\\ -C_f & 0 \end{pmatrix}
= -S_{D_n}\tilde{C}_{D_n}. \]
Since $S_{D_n}$ is symmetric, $\tilde{C}_{D_n}^{\tra}S_{D_n} + S_{D_n}\tilde{C}_{D_n}=0$, therefore $\tilde{C}_{D_n}\in\mathfrak{so}_{2n}(\F_q)$.

The eigenvalue assertion follows because the eigenvalues of $C_f$ are the roots $\lambda_1, \dots, \lambda_n$ of $f(x)$, the eigenvalues of $-C_f^{\tra}$ are $-\lambda_1, \dots, -\lambda_n$, and for type $B_n$ the additional zero block contributes the eigenvalue $0$. This completes the proof.
\end{proof}

\subsection{2-generation via consistent elements and descent}\label{subsec:root-vectors} 

In this subsection, we establish two preliminary results that will be used in the classical case to prove our main results. The first of these is Lemma~\ref{lem:generation-splitting}, which plays a central role in our arguments. This lemma can be viewed as an analogue of~\cite[Lemma~2.1]{Chistopolskaya_sln} for classical Lie algebras over an algebraically closed field of characteristic zero.

\begin{lemma}[Generation by a consistent element and $X$]
\label{lem:generation-splitting}
Let $\g$ be one of $\mathfrak{so}_{2n+1}$, $\mathfrak{sp}_{2n}$, $\mathfrak{so}_{2n}$ over a field $\F$, with root system
$\Phi$, Cartan subalgebra $\h$, and root space decomposition
\[
\g=\h\oplus\bigoplus_{\alpha\in\Phi}\g_\alpha.
\]
Let $T\in\h$ be a consistent element, i.e.\ $\alpha(T)\neq0$ for all $\alpha\in\Phi$ and
$\alpha(T)\neq\beta(T)$ for all distinct $\alpha,\beta\in\Phi$. Let $X\in\g$ have a nonzero component in every root space $\g_\alpha$. Then $T$ and $X$
generate $\g$.
\end{lemma}

\begin{proof}
Write
\[
X=\sum_{\alpha\in\Phi}c_\alpha e_\alpha,
\]
where $e_\alpha\in\g_\alpha$ is nonzero and $c_\alpha\neq0$. Consider
\[
A_k=[X,T,\overset{k}{\dots},T],\qquad k=0,1,\dots,|\Phi|-1.
\]
Since $T\in\h$, we have $[e_\alpha,T]=\alpha(T)e_\alpha$, so
\[
A_k=\sum_{\alpha\in\Phi}c_\alpha(\alpha(T))^k e_\alpha.
\]
In the basis $\{e_\alpha\}_{\alpha\in\Phi}$ of
$\bigoplus_{\alpha\in\Phi}\g_\alpha$, the coordinates of
$A_0,\dots,A_{|\Phi|-1}$ form the matrix
$\diag(c_\alpha)\cdot V$, where
\[
V=((\alpha(T))^k)_{\alpha\in\Phi,\ 0\le k\le|\Phi|-1}
\]
is a Vandermonde matrix. Since $T$ is consistent, the values $\alpha(T)$ are
pairwise distinct, so $V$ is invertible. Since all $c_\alpha\neq0$, the product
$\diag(c_\alpha)V$ is invertible. Hence, $A_0,\dots,A_{|\Phi|-1}$ form a basis
of $\bigoplus_{\alpha\in\Phi}\g_\alpha$.

Taking brackets between $\g_\alpha$ and $\g_{-\alpha}$ yields the Cartan
subalgebra $\h$. Therefore, the Lie algebra generated by $T$ and $X$ contains
$\h\oplus\bigoplus_{\alpha\in\Phi}\g_\alpha=\g$. This completes the proof.
\end{proof}

Lemma~\ref{lem:descent} shows that if a Lie algebra over a field has a set of generators contained in the same type of Lie algebra over a subfield, then the same set serves as a generating set for the Lie algebra over that subfield.
\begin{lemma}[Descent of generation]
\label{lem:descent}
Let $L$ be a finite-dimensional Lie algebra over $\F_q$, let
$L_{\F}=L\otimes_{\F_q}\F$ for a field extension $\F/\F_q$, and let
$x,y\in L$. If $x,y$ generate $L_{\F}$ over $\F$, then they generate $L$ over
$\F_q$.
\end{lemma}

\begin{proof}
Let $M$ be the $\F_q$-Lie subalgebra of $L$ generated by $x$ and $y$, i.e.\ the
$\F_q$-span of all iterated brackets in $x,y$. Then $M\otimes_{\F_q}\F$ is the
$\F$-Lie subalgebra of $L_{\F}$ generated by $x,y$. By hypothesis,
$M\otimes_{\F_q}\F=L_{\F}=L\otimes_{\F_q}\F$. Since $M\subseteq L$ and both
are finite-dimensional over $\F_q$, we get $M=L$. This completes the proof.
\end{proof}

\subsection{Idea used in $2$-generation for exceptional type}\label{sub: Exceptional Type idea}

Let \(\Phi\) be the irreducible root system of rank \(n\) for the simple Lie algebra $\mathfrak{g}$, and $\mathfrak{h}$ be its Cartan subalgebra. Let \(\Phi^+ \subset \Phi\) be a system of positive roots. We fix a basis \(\{h_1, \dots, h_n\}\) of \(\mathfrak{h}\), and let \(x_1, \dots, x_n\) be the dual coordinates on \(\mathfrak{h}\), so that a point
$T= s_1h_1+s_2h_2 + \cdots + s_nh_n$ is identified with the tuple $(s_1, s_2, \cdots, s_n)\in\mathbb F_q^n$ and $x_i(T)=s_i$ for $i=1,2,\dots,n$. Thus, for any root $\alpha$, we have $\alpha(T) = \sum_{i=1}^n\alpha(h_i)\,x_i(T)$. Consequently, for a positive root $\gamma=\sum_{i=1}^n c_i\alpha_i$ with $c_i\in\mathbb Z_{\ge0}$,
the constants are
\[
\gamma(h_j)=\sum_{i=1}^n c_i\,\alpha_i(h_j)=\sum_{i=1}^n c_i\,A_{ij},
\qquad j=1,2, \dots , n.
\]

Now, we define the following polynomial:
\[
f_\mathfrak{g}(x)
\;=\;
\prod_{\alpha \in \Phi^+} \alpha(x)
\;\cdot\;
\prod_{\substack{\alpha,\beta \in \Phi^+ \\ \alpha < \beta}} \bigl(\alpha(x) - \beta(x)\bigr)
\;\cdot\;
\prod_{\substack{\alpha,\beta \in \Phi^+ \\ \alpha \neq \beta}} \bigl(\alpha(x) + \beta(x)\bigr),
\]
Recall the definition~\ref{consistent} that we call an element \(T \in \mathfrak{h}\) consistent if all root evaluations \(\alpha(T)\), \(\alpha \in \Phi\), are pairwise distinct and non-zero. Thus, $T$ is consistent if and only if $f_\mathfrak{g}(T) \ne 0$.

We now recall a powerful combinatorial tool which will be used throughout to prove the $2$-generation result for exceptional type Lie algebras.

\begin{theorem}[Alon's Combinatorial Nullstellensatz \cite{alon}]\label{Alon's theorem}
Let $F$ be a field, and let $g = g(x_1, \dots, x_n)$ be a polynomial in $F[x_1, \dots, x_n]$ of degree $\deg(g) = \sum_{i=1}^n t_i$, where each $t_i$ is a non-negative integer, i.e., the coefficient of the monomial $x_1^{t_1} x_2^{t_2} \cdots x_n^{t_n}$ in $g$ is non-zero. Then, for any subsets $S_1, \dots, S_n \subseteq F$ such that $|S_i| > t_i$ for all $i$, there exists $(s_1, \dots, s_n) \in S_1 \times \cdots \times S_n$ such that
\[
g(s_1, \dots, s_n) \neq 0.
\]
\end{theorem}
\noindent This theorem is used in the lemma below to ensure that consistent elements exist, provided that the field $\F_q$ is sufficiently large and its characteristic does not take certain small values.

\begin{lemma}
\label{lem:consistent}
Let $\mathfrak{g}$ be a finite simple Chevalley Lie algebra and $f_\mathfrak{g}$ be the polynomial defined as above. Let $c_{(i_1,i_2,\dots,i_n)}$ be the coefficient of $x_1^{t_1} \cdots x_n^{t_n}$ in the polynomial $f_\mathfrak{g}$, 
$n' \;:=\; \gcd\bigl\{\, c_{(i_1,i_2,\dots,i_n)} \;\bigm|\; \; c_{(i_1,i_2,\dots,i_n)} \neq 0 \,\bigr\}$ and \(P(n')\) be the set of prime divisors of $n'$. If $\Char(\F_q) \notin P(n')$ and $q > max_{i=1,2,\dots,n}(t_i)$, where $t_i$ is the degree of $x_i$ in a top-degree monomial $x_1^{t_1} \cdots x_n^{t_n}$ of $f_\mathfrak{g}$ with non-zero coefficent over $\F_q$, then there exists a consistent element \(T \in \mathfrak{h}\).
\end{lemma}
\begin{proof}
Since $\Char(\F_q) \notin P(n')$, the polynomial $f_\mathfrak{g}$ is not identically zero over $\F_q$. Let $x_1^{t_1} \cdots x_n^{t_n}$ be a top-degree monomial of $f_\mathfrak{g}$ with non-zero coefficent over $\F_q$ such that $q > max_{i=1,2,\dots,n} (t_i)$. Then for $S_i = \F_q$, by Theorem \ref{Alon's theorem}, we get a point $(s_1,\dots,s_n) \in \mathbb{F}_q^n$ such that $f_\mathfrak{g}(s_1,\dots,s_n) \ne 0$ and hence the element $T = \sum_{i=1}^n s_i\, h_i \in \mathfrak{h}$ is consistent.
\end{proof}

\begin{remark}\label{remark: Degree of Q_phi}
Let $m = |\Phi^+|$. The total degree of $f_\mathfrak{g}$ is
\[
d \;=\; m + \binom{m}{2} + m(m-1) \;=\; \frac{m(3m-1)}{2}.
\]
In terms of the total number of roots \(|\Phi| = 2m\), this is
\[
d \;=\; \frac{|\Phi|(3|\Phi| - 2)}{8}.
\]
\end{remark}


\section{Rationality of certain matrices}\label{Section: Rationality of matrices}

We begin by defining an element $X \in \mathfrak{g}$ having a non-zero component in every root space $\mathfrak{g}_{\alpha}$, whose conjugates will serve as one of the generators for classical Lie algebras $\mathfrak{g}$. Together with the appropriate conjugate of $\tilde C$ defined in subsection~\ref{subsec:companion} for each classical type, these elements will serve as the generators. However, we need to find an appropriate conjugating element $P$ which, after conjugation, simultaneously brings these two matrices into our Lie algebra. In this section, we carry out these technical computations for each classical Lie algebra over finite fields.

For type $C_n$ define
\[
\begin{aligned}
X_{C_n}
&= \sum_{1\le i\ne j\le n}(E_{i,j}-E_{-j,-i})
 + \sum_{1\le i<j\le n}(E_{i,-j}+E_{j,-i}) \\
&\quad + \sum_{1\le i<j\le n}(E_{-i,j}+E_{-j,i})
 + \sum_{i=1}^{n}E_{i,-i}
 + \sum_{i=1}^{n}E_{-i,i}.
\end{aligned}
\]
This is the sum of one nonzero root vector from every root space, so
$X_{C_n}\in\mathfrak{sp}_{2n}(\F_q)$ has nonzero component in every root space.

For types $B_n$ and $D_n$, let $\mathbf B=(b_{ij})\in M_n(\F_q)$ be
skew-symmetric, with $b_{ii}=0$ and $b_{ji}=-b_{ij}$. Define
\[
X_{D_n}(\mathbf B)
=
\sum_{1\le i\ne j\le n}(E_{i,j}-E_{-j,-i})
+\sum_{1\le i<j\le n}b_{ij}\bigl(E_{i,-j}-E_{j,-i}\bigr)
+\sum_{1\le i<j\le n}b_{ij}\bigl(E_{-i,j}-E_{-j,i}\bigr),
\]
and
\[
\begin{aligned}
X_{B_n}(\mathbf B)
&=
\sum_{1\le i\ne j\le n}(E_{i,j}-E_{-j,-i})
+\sum_{1\le i<j\le n}b_{ij}\bigl(E_{i,-j}-E_{j,-i}\bigr)
+\sum_{1\le i<j\le n}b_{ij}\bigl(E_{-i,j}-E_{-j,i}\bigr)\\
&\quad
+\sum_{i=1}^{n}(E_{i,0}-E_{0,-i})
+\sum_{i=1}^{n}(E_{-i,0}-E_{0,i}).
\end{aligned}
\]
Then $X_{D_n}(\mathbf B)\in\mathfrak{so}_{2n}(\F_q)$ and
$X_{B_n}(\mathbf B)\in\mathfrak{so}_{2n+1}(\F_q)$.

Writing matrices in blocks of sizes $n,n$ (respectively $n,1,n$), we have
\[
X_{C_n}=
\begin{pmatrix}
\mathbf 1 & \mathbf S\\
\mathbf S & -\mathbf 1
\end{pmatrix},
\qquad
X_{D_n}(\mathbf B)=
\begin{pmatrix}
\mathbf 1 & \mathbf B\\
\mathbf B & -\mathbf 1
\end{pmatrix}, \qquad
X_{B_n}(\mathbf B)=
\begin{pmatrix}
\mathbf 1 & e & \mathbf B\\
-e^{\tra} & 0 & -e^{\tra}\\
\mathbf B & e & -\mathbf 1
\end{pmatrix},
\]
where $e=(1,\dots,1)^{\tra},\qquad
\mathbf 1=ee^{\tra}-I,\quad \text{and} \quad
\mathbf S=ee^{\tra}$.

\subsection{Rationality for $C_n$}\label{subsec:rational-Cn} 

We now prove the rationality lemma, which is the key step in proving the main result in the subsequent section via the descent technique from the generators of the Lie algebra over $\F_{q^n}$ to $\F_q$. The matrix $P$ is constructed explicitly from the left eigenvectors of $\widetilde{C}_{C_n}$ described above; the proof that this $P$ lies in $\mathrm{Sp}_{2n}(\F)$ and satisfies $P^{-1}X_{C_n}P\in\mathfrak{sp}_{2n}(\F_q)$ is given below.

\begin{lemma}[Rationality for $C_n$]
\label{lem:rational-Cn}
Let $p\neq 2$ and $n\ge 2$. Let $f\in\F_q[x]$ be the minimal polynomial of a normal element $\lambda \in \F_{q^n}$ over $\F_q$, and let $\tilde{C}_{C_n}$ be
the companion block matrix of Lemma~\ref{lem:companion-blocks}. Let $\F=\F_{q^n}$ be a splitting field for $f$, and let \[ T = \diag(t_1,\dots,t_n,-t_1,\dots,-t_n),
\] 
where $t_i = \Fr^{i-1}(\lambda)$. Let $X_{C_n} \in \mathfrak{sp}_{2n}(\F_q)$ be as above. Then there exists
$P\in \Sp_{2n}(\F)$ with $P\tilde{C}_{C_n}P^{-1} = T$ such that
\[
P^{-1}X_{C_n}P\in\mathfrak{sp}_{2n}(\F_q).
\]
\end{lemma}

\begin{proof}
Set
\[
Q=\begin{pmatrix}u_1\\ \vdots\\ u_n\end{pmatrix},
\qquad
u_i=(1,t_i,t_i^2,\dots,t_i^{n-1}),\qquad i=1,\dots,n,
\]
so that $Q_{ik}=\Fr^{i-1}(\lambda^{k-1})$. Since the $t_i$ are distinct, $Q$ is
invertible. Define $P=\begin{pmatrix}(Q^{-1})^{\tra}&0\\0&Q\end{pmatrix}\in GL_{2n}(\F)$. Since $f(t_i)=0$, the Vandermonde vector $u_i^{\tra}$ is a right eigenvector of
$C_f$ with eigenvalue $t_i$: $C_fu_i^{\tra}=t_iu_i^{\tra}$.
Transposing gives $u_iC_f^{\tra}=t_iu_i$. Writing these equations for all $i$ in matrix form yields
\[
QC_f^{\tra}=DQ,\qquad D=\diag(t_1,\dots,t_n).
\]
Multiplying on the left by $Q^{-1}$ gives $C_f^{\tra}Q^{-1}=Q^{-1}D$. Taking transposes and using $D^{\tra}=D$, we obtain $(Q^{-1})^{\tra}C_f=D(Q^{-1})^{\tra}$. Thus, the rows of $(Q^{-1})^{\tra}$ are left eigenvectors of $C_f$ with eigenvalues $t_1,\dots,t_n$. On the other hand, $u_i(-C_f^{\tra})=-t_iu_i$, so the rows of $Q$ are left eigenvectors of $-C_f^{\tra}$ with eigenvalues
$-t_1,\dots,-t_n$. The left eigenvectors of
\[
\widetilde{C}_n=
\begin{pmatrix}
C_f & 0\\
0 & -C_f^{\mathsf T}
\end{pmatrix}
\]
are therefore $(\ell_i,0)$ for the eigenvalue $t_i$, $(0,u_i)$ for the eigenvalue $-t_i$, where \(\ell_i\) is the \(i\)-th row of \((Q^{-1})^{\mathsf T}\) and \(u_i\) is the \(i\)-th row of \(Q\). Hence, we get
\[
P=
\begin{pmatrix}
(Q^{-1})^{\mathsf T} & 0\\
0 & Q
\end{pmatrix}.
\]
Thus, with the convention that the rows of \(P\) are left eigenvectors, $P\widetilde{C}_nP^{-1} =
\operatorname{diag}(t_1,\dots,t_n,-t_1,\dots,-t_n)$.
Therefore, we have
\[
P\tilde{C}_{C_n}P^{-1}
=
\begin{pmatrix}(Q^{-1})^{\tra}&0\\0&Q\end{pmatrix}
\begin{pmatrix}C_f&0\\0&-C_f^{\tra}\end{pmatrix}
\begin{pmatrix}Q^{\tra}&0\\0&Q^{-1}\end{pmatrix}
=
\begin{pmatrix}D&0\\0&-D\end{pmatrix}
=T.
\]
Using $Q^{\tra}(Q^{-1})^{\tra}=I$ and $Q^{-1}Q=I$, we compute
\[
P^{\tra}S_{C_n}P
=
\begin{pmatrix}Q^{-1}&0\\0&Q^{\tra}\end{pmatrix}
\begin{pmatrix}0&I\\-I&0\end{pmatrix}
\begin{pmatrix}(Q^{-1})^{\tra}&0\\0&Q\end{pmatrix}
=
\begin{pmatrix}0&I\\-I&0\end{pmatrix}
=S_{C_n},
\]
so $P\in\Sp_{2n}(\F)$. By definition of $X_{C_n}$,
\[
X_{C_n}=\begin{pmatrix}\mathbf 1&\mathbf S\\ \mathbf S&-\mathbf 1\end{pmatrix},
\qquad
\mathbf 1=ee^{\tra}-I,\quad \mathbf S=ee^{\tra},
\]
where $e=(1,\dots,1)^{\tra}\in\F^n$. Since
$P^{-1}=\begin{pmatrix}Q^{\tra}&0\\0&Q^{-1}\end{pmatrix}$, we obtain
\[
P^{-1}X_{C_n}P
=
\begin{pmatrix}
Q^{\tra}\mathbf 1(Q^{-1})^{\tra}&Q^{\tra}\mathbf SQ\\
Q^{-1}\mathbf S(Q^{-1})^{\tra}&-Q^{-1}\mathbf 1Q
\end{pmatrix}.
\]
Substituting $\mathbf 1=ee^{\tra}-I$ and $\mathbf S=ee^{\tra}$, and using
$Q^{\tra}(Q^{-1})^{\tra}=I$ and $Q^{-1}Q=I$, the four blocks become $Q^{\tra}\mathbf 1(Q^{-1})^{\tra}
=(Q^{\tra}e)(e^{\tra}(Q^{-1})^{\tra})-I, Q^{\tra}\mathbf SQ=(Q^{\tra}e)(e^{\tra}Q), Q^{-1}\mathbf S(Q^{-1})^{\tra}
=(Q^{-1}e)(e^{\tra}(Q^{-1})^{\tra})=(Q^{-1}e)(Q^{-1}e)^{\tra}, -Q^{-1}\mathbf 1Q=-(Q^{-1}e)(e^{\tra}Q)+I$.
Thus, every block is a product of the vectors,
\[
u=Q^{-1}e,\qquad v=e^{\tra}Q,
\]
together with their transposes and the identity matrix.

We now show that $u,v\in\F_q^n$. The first column of $Q$ is $(\Fr^{i-1}(\lambda^0))_{i=1}^n=(1,\dots,1)^{\tra}=e$.
Hence, we get $Qe_1=e$, where $e_1=(1,0,\dots,0)^{\tra}$. Since $Q$ is invertible, $Q^{-1}e=e_1\in\F_q^n$,
so $u=e_1\in\F_q^n$ and $u^{\tra}=e_1^{\tra}\in\F_q^n$. Next, the $j$-th entry of $v=e^{\tra}Q$ is
\[
v_j=\sum_{i=1}^n Q_{ij}
=\sum_{i=1}^n \Fr^{i-1}(\alpha^{j-1})
=\Tr(\alpha^{j-1})\in\F_q.
\]
Therefore $v\in\F_q^n$, and hence $v^{\tra}=Q^{\tra}e\in\F_q^n$. Consequently, every block of $P^{-1}X_{C_n}P$ has all entries in $\F_q$, so $P^{-1}X_{C_n}P\in M_{2n}(\F_q)$. Since $P\in\Sp_{2n}(\F)$ and $X_{C_n}\in\mathfrak{sp}_{2n}(\F)$, the identity
$P^{\tra}S_{C_n}P=S_{C_n}$ gives
\[
(P^{-1}X_{C_n}P)^{\tra}S_{C_n}+S_{C_n}(P^{-1}X_{C_n}P)
=
P^{\tra}(X_{C_n}^{\tra}S_{C_n}+S_{C_n}X_{C_n})P=0.
\]
Combining this with $P^{-1}X_{C_n}P\in M_{2n}(\F_q)$,
\[
P^{-1}X_{C_n}P\in\mathfrak{sp}_{2n}(\F)\cap M_{2n}(\F_q)
=\mathfrak{sp}_{2n}(\F_q).
\]
This completes the proof of Lemma~\ref{lem:rational-Cn}.
\end{proof}

\subsection{Rationality for $B_n$ and $D_n$}
\label{subsec:rational-Bn-Dn} 

Throughout this subsection, assume $p\neq2$. Let $\lambda\in\F_{q^n}$ be a normal element over $\F_q$, put $t_i=\Fr^{i-1}(\lambda)$, and let $Q=(Q_{ik}), Q_{ik}=\Fr^{i-1}(\lambda^{k-1}), 1\le i,k\le n$ be the associated Vandermonde matrix; $Q$ is invertible. Let
$\mathbf B=(b_{ij})\in M_n(\F_q)$ be skew-symmetric and circulant, with
$b_{ii}=0$ and $b_{ji}=-b_{ij}$. Let $X_{D_n}(\mathbf B)$ and
$X_{B_n}(\mathbf B)$ be as in Subsection~\ref{subsec:root-vectors}. Set
\[
P_{D_n}=
\begin{pmatrix}
(Q^{-1})^{\tra}&0\\
0&Q
\end{pmatrix}
\quad \text{and} \quad
P_{B_n}=
\begin{pmatrix}
(Q^{-1})^{\tra}&0&0\\
0&1&0\\
0&0&Q
\end{pmatrix}.
\]
From $QC_f^{\tra}=DQ$ with $D=\diag(t_1,\dots,t_n)$, we get
\[
P_{D_n}\tilde C_{D_n}P_{D_n}^{-1}
=\diag(t_1,\dots,t_n,-t_1,\dots,-t_n),
\]
\[
P_{B_n}\tilde C_{B_n}P_{B_n}^{-1}
=\diag(t_1,\dots,t_n,0,-t_1,\dots,-t_n).
\]
Using $Q^{\tra}(Q^{-1})^{\tra}=I$ and $Q^{-1}Q=I$, one obtains $P_{D_n}^{\tra}S_{D_n}P_{D_n}=S_{D_n}, 
P_{B_n}^{\tra}S_{B_n}P_{B_n}=S_{B_n}$, so $P_{D_n}\in O_{2n}(\F_{q^n})$ and
$P_{B_n}\in O_{2n+1}(\F_{q^n})$.

\begin{lemma}[Frobenius invariance of the conjugated blocks]
\label{lem:frob-blocks}
Let $\lambda\in\F_{q^n}$ be normal over $\F_q$, let
$Q_{ik}=\Fr^{i-1}(\lambda^{k-1})$, and let $\mathbf B\in M_n(\F_q)$ be
circulant and skew-symmetric. Then
\[
Q^{\tra}\mathbf BQ\in M_n(\F_q),
\qquad
Q^{-1}\mathbf B(Q^{-1})^{\tra}\in M_n(\F_q),
\]
and $Q^{\tra}e, e^{\tra}Q, Q^{-1}e, e^{\tra}(Q^{-1})^{\tra}
\in\F_q^n$. Consequently, 
\[
Q^{\tra}\mathbf 1(Q^{-1})^{\tra},\quad
Q^{-1}\mathbf 1Q,\quad
Q^{-1}\mathbf 1(Q^{-1})^{\tra},\quad
Q^{\tra}\mathbf 1Q
\in M_n(\F_q),
\]
where $\mathbf 1=ee^{\tra}-I$.
\end{lemma}

\begin{proof}
A matrix over $\F_{q^n}$ lies in $M_n(\F_q)$ if and only if it is fixed by the
entrywise Frobenius map $M\mapsto M^{\Fr}$.

Since $Q_{ik}=\Fr^{i-1}(\lambda^{k-1})$, $(Q^{\Fr})_{ik}=(Q_{ik})^q=\Fr^{i}(\lambda^{k-1})=Q_{i+1,k}$, indices modulo $n$, so $Q^{\Fr}=RQ$ with $R$ the cyclic shift. Hence $(Q^{-1})^{\Fr}=Q^{-1}R^{-1}, \bigl((Q^{-1})^{\tra}\bigr)^{\Fr} = R(Q^{-1})^{\tra}, (Q^{\tra})^{\Fr}=Q^{\tra}R^{-1}$, using $R^{\tra}=R^{-1}$.

Since $\mathbf B\in M_n(\F_q)$ is circulant, it is a polynomial in $R$. Thus, $\mathbf B^{\Fr}=\mathbf B$ and $R^{-1}\mathbf B R=\mathbf B$. Therefore, we have
\[
(Q^{\tra}\mathbf BQ)^{\Fr}
=(Q^{\tra})^{\Fr}\mathbf B^{\Fr}Q^{\Fr}
=Q^{\tra}R^{-1}\mathbf B\,RQ
=Q^{\tra}\mathbf BQ.
\]
With $S=Q^{-1}$, so $S^{\Fr}=SR^{-1}$ and
$(S^{\tra})^{\Fr}=RS^{\tra}$, $(S\mathbf B S^{\tra})^{\Fr}
=S^{\Fr}\mathbf B^{\Fr}(S^{\tra})^{\Fr}
=SR^{-1}\mathbf B\,RS^{\tra} = S\mathbf B S^{\tra}$.
Both blocks are fixed by $\Fr$, hence lie in $M_n(\F_q)$.
The first column of $Q$ is $e$, so $Qe_1=e$ and
$Q^{-1}e=e_1\in\F_q^n$; transposing gives
$e^{\tra}(Q^{-1})^{\tra}=e_1^{\tra}\in\F_q^n$. Also
\[
(e^{\tra}Q)_j=\sum_{i=1}^{n}\Fr^{i-1}(\lambda^{j-1})
=\Tr(\lambda^{j-1})\in\F_q,
\]
so $e^{\tra}Q\in\F_q^n$ and $Q^{\tra}e=(e^{\tra}Q)^{\tra}\in\F_q^n$.
Finally $\mathbf 1=ee^{\tra}-I$, so each block involving $\mathbf 1$ is an
outer product of two of these vectors, plus $\pm I$, hence lies in
$M_n(\F_q)$.
\end{proof}

\begin{lemma}[Rationality for $B_n$ and $D_n$]
\label{lem:rational-BD}
Let $n\ge2$ for $B_n$ and $n\ge3$ for $D_n$. Let $f\in\F_q[x]$ be the minimal
polynomial of a normal element $\lambda \in\F_{q^n}$ over $\F_q$, let
$\tilde C_{B_n}$ and $\tilde C_{D_n}$ be the companion block matrices of
Lemma~\ref{lem:companion-blocks}, and let $T_{D_n}=\diag(t_1,\dots,t_n,-t_1,\dots,-t_n), T_{B_n}=\diag(t_1,\dots,t_n,0,-t_1,\dots,-t_n)$, where $t_i=\Fr^{i-1}(\lambda)$. Let $\mathbf B\in M_n(\F_q)$ be circulant and
skew-symmetric, and let $X_{D_n}(\mathbf B)$ and $X_{B_n}(\mathbf B)$ be as
above. Then there exist $P_{D_n}\in O_{2n}(\F_{q^n})$ and
$P_{B_n}\in O_{2n+1}(\F_{q^n})$ such that
\[
P_{D_n}\tilde C_{D_n}P_{D_n}^{-1}=T_{D_n},
\qquad
P_{B_n}\tilde C_{B_n}P_{B_n}^{-1}=T_{B_n},
\]
and
\[
P_{D_n}^{-1}X_{D_n}(\mathbf B)P_{D_n}\in\mathfrak{so}_{2n}(\F_q),
\qquad
P_{B_n}^{-1}X_{B_n}(\mathbf B)P_{B_n}\in\mathfrak{so}_{2n+1}(\F_q).
\]
\end{lemma}

\begin{proof}
Take $P_{D_n}$ and $P_{B_n}$ as above. Block multiplication gives
\[
P_{D_n}^{-1}X_{D_n}(\mathbf B)P_{D_n}
=
\begin{pmatrix}
Q^{\tra}\mathbf 1(Q^{-1})^{\tra}&Q^{\tra}\mathbf BQ\\
Q^{-1}\mathbf B(Q^{-1})^{\tra}&-Q^{-1}\mathbf 1Q
\end{pmatrix},
\]
and
\[
P_{B_n}^{-1}X_{B_n}(\mathbf B)P_{B_n}
=
\begin{pmatrix}
Q^{\tra}\mathbf 1(Q^{-1})^{\tra}&Q^{\tra}e&Q^{\tra}\mathbf BQ\\
-e^{\tra}(Q^{-1})^{\tra}&0&-e^{\tra}Q\\
Q^{-1}\mathbf B(Q^{-1})^{\tra}&Q^{-1}e&-Q^{-1}\mathbf 1Q
\end{pmatrix}.
\]
By Lemma~\ref{lem:frob-blocks}, every block lies in $M_n(\F_q)$ or $\F_q^n$.
Hence both conjugated matrices lie in $M_{2n}(\F_q)$ and $M_{2n+1}(\F_q)$,
respectively. Since the conjugating matrices are orthogonal over $\F_{q^n}$
and $X_{D_n}(\mathbf B),X_{B_n}(\mathbf B)$ are skew, the identities
$P^{\tra}SP=S$ give
\[
(P^{-1}XP)^{\tra}S+S(P^{-1}XP)=0.
\]
Therefore, the conjugated matrices lie in $\mathfrak{so}_{2n}(\F_q)$ and
$\mathfrak{so}_{2n+1}(\F_q)$.
\end{proof}

\begin{remark}\label{rem:consistent}
If $f(x)$ is the minimal polynomial of a normal element $\lambda \in \F_{q^n}$ over $\F_q$, then $t_i=\Fr^{i-1}(\lambda)$ are linearly independent over $\F_q$. In
particular $t_i \neq 0$ and $t_i \neq \pm t_j$ for $i\neq j$. Since $p\neq2$, every equality between two root evaluations $\alpha(T)$ for the diagonal matrices $T$ above (see Lemmas~\ref{lem:rational-Cn} and~\ref{lem:rational-BD}) gives a nontrivial $\F_q$-linear relation among the $t_i$, which is impossible. Hence, $T$ is consistent, i.e.,  the values $\alpha(T)$, $\alpha\in\Phi$, are pairwise distinct and nonzero.
\end{remark}


\section{Classical type split simple Chevalley Lie algebras over $\mathbb{F}_q$}\label{section:classical}

Throughout this section, let $\mathbb{F}_q$ be a finite field of odd characteristic, and let $n\geq 1$. The split Chevalley Lie algebra of type $A_n$, namely $\mathfrak{sl}_{n+1}(\mathbb{F}_q)$, is two-generated (see~\cite{cjz}). In this section, we investigate the two-generation property for the remaining split classical types $B_n$, $C_n$, and $D_n$. We consider their matrix realizations as $\mathfrak{so}_{2n+1}(\mathbb{F}_q)$, $\mathfrak{sp}_{2n}(\mathbb{F}_q)$, and $\mathfrak{so}_{2n}(\mathbb{F}_q)$, respectively, with respect to the standard nondegenerate bilinear forms $S_{B_n}$, $S_{C_n}$, and $S_{D_n}$ defined earlier (see Section~\ref{Section: prelim}). Since $p\neq 2$, these matrix realizations coincide with the corresponding split Chevalley Lie algebras. We now state the main result of this section.

\begin{theorem}[2-generation for classical Chevalley Lie algebras]\label{thm:two-generation classical}
Let $p\neq 2$ and $n\ge 1$. The split simple Chevalley Lie algebras $\mathfrak{so}_{2n+1}(\F_q)$ and $\mathfrak{sp}_{2n}(\F_q)$ are $2$-generated. For $n\ge 3$, the split simple Chevalley Lie algebra $\mathfrak{so}_{2n}(\F_q)$ is 2-generated.
\end{theorem}
\noindent We prove the theorem case by case by explicitly constructing the required generating elements. More precisely, for each classical Lie algebra, we first construct a consistent element $T \in \mathfrak{h}$ and an element $X \in \mathfrak{g}$ having a non-zero component in every root space $\mathfrak{g}_{\alpha}$ over a suitable field extension $\mathbb{F}$ of $\mathbb{F}_q$. We then employ a descent technique to obtain generators over $\mathbb{F}_q$ by conjugating the generators over $\mathbb{F}$ with a matrix $P$. The matrix $P$ is constructed explicitly using the left eigenvectors of $\widetilde{C}_{B_n}$, $\widetilde{C}_{C_n}$, and $\widetilde{C}_{D_n}$, as described in Section~\ref{Section: Rationality of matrices}.

\subsection{The odd-rank case of $B_n$ and $D_n$}
\label{sec:odd-BD}

In this subsection, we explicitly establish the two-generation property for the classical Lie algebras of types $B_n$ and $D_n$ when $n$ is odd. This requires constructing a suitable matrix $\mathbf{B}$ in each of $X_{B_n}$ and $X_{D_n}$ so that the resulting elements have non-zero components in every root space of the respective Lie algebras (see Section~\ref{Section: Rationality of matrices}). Assume $n\ge3$ is odd and put $m=(n-1)/2$. For $1\le i<j\le n$, set
\[
b_{ij}=
\begin{cases}
+1,&1\le j-i\le m,\\[1mm]
-1,&m<j-i\le n-1,
\end{cases}
\qquad
b_{ji}:=-b_{ij},\qquad b_{ii}:=0.
\]
Then
\[
\mathbf B_{\mathrm{odd}}
=\sum_{1\le i<j\le n}b_{ij}(F_{ij}-F_{ji})
=\sum_{k=1}^{m}(R^k-R^{-k})
\in M_n(\F_q)
\]
is circulant and skew-symmetric, and all off-diagonal entries are nonzero. Hence, the product matrices $X_{D_n}\mathbf B_{\mathrm{odd}}$ and
$X_{B_n}\mathbf B_{\mathrm{odd}}$ have nonzero component in every root space.

\begin{proposition}
[Two-generation for odd-rank $B_n$ and $D_n$]
\label{prop:two-gen-odd}
Let $p\neq2$ and let $n\ge3$ be odd. Choose
$P_{D_n}\in O_{2n}(\F_{q^n})$ and
$P_{B_n}\in O_{2n+1}(\F_{q^n})$ as in Lemma~\ref{lem:rational-BD}, applied to
$\mathbf B=\mathbf B_{\mathrm{odd}}$. Then
\[
\bigl\langle \tilde C_{D_n}, \ P_{D_n}^{-1} (X_{D_n}\mathbf B_{\mathrm{odd}})P_{D_n}
\bigr\rangle_{\F_q}
=
\mathfrak{so}_{2n}(\F_q),
\]
and
\[
\bigl\langle \tilde C_{B_n}, \ P_{B_n}^{-1} (X_{B_n}\mathbf B_{\mathrm{odd}})P_{B_n}
\bigr\rangle_{\F_q}
=
\mathfrak{so}_{2n+1}(\F_q).
\]
\end{proposition}

\begin{proof}
Let us simply write $\F=\F_{q^n}$ and the diagonal elements 
$T_{D_n}=\diag(t_1,\dots,t_n,-t_1,\dots,-t_n)$ and $ 
T_{B_n}=\diag(t_1,\dots,t_n,0,-t_1,\dots,-t_n)$. By Lemma~\ref{lem:rational-BD}, after replacing $t_i$ by conjugates if necessary, we may assume $P_{D_n}\tilde C_{D_n}P_{D_n}^{-1}=T_{D_n}$, $P_{B_n}\tilde C_{B_n}P_{B_n}^{-1}=T_{B_n}$, 
and
\[
Y_{D_n}:=P_{D_n}^{-1}X_{D_n}\mathbf B_{\mathrm{odd}}P_{D_n}\in \mathfrak{so}_{2n}(\F_q),
\]
\[
Y_{B_n}:=P_{B_n}^{-1}X_{B_n}\mathbf B_{\mathrm{odd}}  P_{B_n} \in\mathfrak{so}_{2n+1}(\F_q).
\]
By Remark~\ref{rem:consistent}, $T_{D_n}$ and $T_{B_n}$ are consistent. Since all off-diagonal entries of $\mathbf B_{\mathrm{odd}}$ are nonzero,
$X_{D_n}\mathbf B_{\mathrm{odd}}$ and
$X_{B_n}\mathbf B_{\mathrm{odd}}$ have nonzero component in every root space relative to $T_{D_n}$ and $T_{B_n}$, respectively. Hence Lemma~\ref{lem:generation-splitting} gives
\[
\bigl\langle T_{D_n}, X_{D_n}\mathbf B_{\mathrm{odd}}\bigr\rangle_{\F}
=
\mathfrak{so}_{2n}(\F),
\qquad
\bigl\langle T_{B_n}, X_{B_n}\mathbf B_{\mathrm{odd}} \bigr\rangle_{\F}
=
\mathfrak{so}_{2n+1}(\F).
\]
Conjugation by $P_{D_n}^{-1}$ (respectively $P_{B_n}^{-1}$) is a Lie
algebra automorphism of the full matrix algebra which preserves
$\mathfrak{so}_{2n}(\F)$ (respectively $\mathfrak{so}_{2n+1}(\F)$), because
$P_{D_n}^{\tra}S_{D_n}P_{D_n}=S_{D_n}$ and
$P_{B_n}^{\tra}S_{B_n}P_{B_n}=S_{B_n}$. It sends
$T_{D_n} = P_{D_n}\tilde C_{D_n}P_{D_n}^{-1} \mapsto \tilde C_{D_n}$, and $ X_{D_n}\mathbf B_{\mathrm{odd}}\mapsto Y_{D_n}$, and
$T_{B_n} = P_{B_n}\tilde C_{B_n}P_{B_n}^{-1}
\mapsto \tilde C_{B_n}, X_{B_n}\mathbf B_{\mathrm{odd}}\mapsto Y_{B_n}$. Therefore
\[
\bigl\langle \tilde C_{D_n},Y_{D_n}\bigr\rangle_{\F}
=
\mathfrak{so}_{2n}(\F),
\qquad
\bigl\langle \tilde C_{B_n},Y_{B_n}\bigr\rangle_{\F}
=
\mathfrak{so}_{2n+1}(\F).
\]
The elements $\tilde C_{D_n},Y_{D_n}$ lie in
$\mathfrak{so}_{2n}(\F_q)$, and
$\tilde C_{B_n},Y_{B_n}$ lie in $\mathfrak{so}_{2n+1}(\F_q)$. By
Lemma~\ref{lem:descent}, generation over $\F$ implies generation over
$\F_q$. Hence
\[
\bigl\langle \tilde C_{D_n},Y_{D_n}\bigr\rangle_{\F_q}
=
\mathfrak{so}_{2n}(\F_q),
\qquad
\bigl\langle \tilde C_{B_n},Y_{B_n}\bigr\rangle_{\F_q}
=
\mathfrak{so}_{2n+1}(\F_q),
\]
which is the claim. This completes the proof of Proposition~\ref{prop:two-gen-odd}.
\end{proof}

\subsection{The even-rank case of $B_n$ and $D_n$}\label{sec:even-BD}
In this subsection, we explicitly establish the two-generation property for the classical Lie algebras of types $B_n$ and $D_n$ when $n$ is even. As in the odd-dimensional case, this requires constructing a suitable matrix $\mathbf{B}$ in each of $X_{B_n}$ and $X_{D_n}$ so that the resulting elements have non-zero components in every root space of the respective Lie algebras (see Section~\ref{Section: Rationality of matrices}). Assume $n=2m$. For $D_n$ assume $n\ge4$; for $B_n$ assume $n\ge2$, with
$B_2$ treated separately below. For $1\le i<j\le n$, set
\[
b_{ij}=
\begin{cases}
+1,&1\le j-i\le m-1,\\[1mm]
0,&j-i=m,\\[1mm]
-1,&m+1\le j-i\le n-1,
\end{cases}
\qquad
b_{ji}:=-b_{ij},\qquad b_{ii}:=0.
\]
Equivalently,
\[
\mathbf B_{\mathrm{even}}
=\sum_{k=1}^{m-1}(R^k-R^{-k})
\in M_n(\F_q).
\]
Then $\mathbf B_{\mathrm{even}}$ is circulant and skew-symmetric, and
\[
(\mathbf B_{\mathrm{even}})_{i,i+m}=0,\qquad
(\mathbf B_{\mathrm{even}})_{ij}\neq0\quad\text{for all other }i\neq j.
\]
Thus, the product matrices $X_{D_n}\mathbf B_{\mathrm{even}}$ and
$X_{B_n}\mathbf B_{\mathrm{even}}$ have nonzero components in all root spaces
except
\[
\pm(\varepsilon_i+\varepsilon_{i+m}),\qquad 1\le i\le m.
\]

\begin{proposition}
[Two-generation for even $B_n$ and $D_n$]
\label{prop:two-gen-even}
Let $p\neq2$ and $n=2m$. Choose
$P_{D_n}\in O_{2n}(\F_{q^n})$ and
$P_{B_n}\in O_{2n+1}(\F_{q^n})$ as in Lemma~\ref{lem:rational-BD}, applied to
$\mathbf B=\mathbf B_{\mathrm{even}}$ (when the corresponding type is defined).
Then the following hold.
\begin{enumerate}
\item If $n\ge4$, then
\[
\bigl\langle \tilde C_{D_n},
P_{D_n}^{-1}X_{D_n}(\mathbf B_{\mathrm{even}})P_{D_n}
\bigr\rangle_{\F_q}
=
\mathfrak{so}_{2n}(\F_q).
\]
\item If $n\ge4$, then
\[
\bigl\langle \tilde C_{B_n},
P_{B_n}^{-1}X_{B_n}(\mathbf B_{\mathrm{even}})P_{B_n}
\bigr\rangle_{\F_q}
=
\mathfrak{so}_{2n+1}(\F_q).
\]
\item For $B_2$,
\[
\bigl\langle \tilde C_{B_2},
P_{B_2}^{-1}X_{B_2}(\mathbf B_{\mathrm{even}})P_{B_2}
\bigr\rangle_{\F_q}
=
\mathfrak{so}_{5}(\F_q).
\]
\end{enumerate}
\end{proposition}

\begin{proof}
Put $\F=\F_{q^n}$ and let $P=P_{D_n}$ or $P=P_{B_n}$. Set
\[
T_{D_n}=\diag(t_1,\dots,t_n,-t_1,\dots,-t_n),
\qquad
T_{B_n}=\diag(t_1,\dots,t_n,0,-t_1,\dots,-t_n).
\]
By Lemma~\ref{lem:rational-BD}, $P\tilde C_{D_n}P^{-1}=T_{D_n}, P\tilde C_{B_n}P^{-1}=T_{B_n}$, and
\[
Y_{D_n}:=P^{-1}X_{D_n}(\mathbf B_{\mathrm{even}})P
\in\mathfrak{so}_{2n}(\F_q),
\qquad
Y_{B_n}:=P^{-1}X_{B_n}(\mathbf B_{\mathrm{even}})P
\in\mathfrak{so}_{2n+1}(\F_q).
\]
By Remark~\ref{rem:consistent}, the diagonal matrices $T_{D_n}$ and $T_{B_n}$
are consistent. Hence, by applying powers of $\ad T$ and interpolating the
distinct eigenvalues, we can isolate from
$X_{D_n}(\mathbf B_{\mathrm{even}})$ and
$X_{B_n}(\mathbf B_{\mathrm{even}})$ their components in every root space in
which these components are nonzero.

By construction, $X_{D_n}(\mathbf B_{\mathrm{even}})$ and
$X_{B_n}(\mathbf B_{\mathrm{even}})$ have nonzero components in all root
spaces except
\[
\pm(\varepsilon_i+\varepsilon_{i+m}),\qquad 1\le i\le m.
\]
All other root vectors therefore lie in the Lie algebra generated by the
corresponding pair $(T,X)$. For $n\ge4$, choose
$k\notin\{i,i+m\}$. Then
\[
[e_{\varepsilon_i-\varepsilon_k},\,
e_{\varepsilon_k+\varepsilon_{i+m}}]
=
c\,e_{\varepsilon_i+\varepsilon_{i+m}},
\qquad c\neq0,
\]
where the brackets denote the corresponding root vectors. Since
$e_{\varepsilon_i-\varepsilon_k}$ and
$e_{\varepsilon_k+\varepsilon_{i+m}}$ are already generated, the missing root
vectors $\pm e_{\varepsilon_i+\varepsilon_{i+m}}$ are generated as well.
Hence all root vectors lie in the generated subalgebra, and therefore
\[
\bigl\langle T_{D_n},X_{D_n}(\mathbf B_{\mathrm{even}})\bigr\rangle_{\F}
=
\mathfrak{so}_{2n}(\F),
\qquad
\bigl\langle T_{B_n},X_{B_n}(\mathbf B_{\mathrm{even}})\bigr\rangle_{\F}
=
\mathfrak{so}_{2n+1}(\F).
\]

For $B_2$, we have $m=1$ and $\mathbf B_{\mathrm{even}}=0$. The element
$X_{B_2}(\mathbf B_{\mathrm{even}})$ still has nonzero components in the short
root spaces $\pm\varepsilon_1,\pm\varepsilon_2$. The missing root vector
$e_{\varepsilon_1+\varepsilon_2}$ is obtained from
\[
[e_{\varepsilon_1},e_{\varepsilon_2}]
=
c\,e_{\varepsilon_1+\varepsilon_2},
\qquad c\neq0,
\]
and similarly for its negative. Thus
\[
\bigl\langle T_{B_2},X_{B_2}(\mathbf B_{\mathrm{even}})\bigr\rangle_{\F}
=
\mathfrak{so}_{5}(\F).
\]

Conjugation by $P^{-1}$ is a Lie algebra automorphism of the full matrix
algebra preserving the corresponding orthogonal Lie algebra. It sends
\[
T_{D_n}=P\tilde C_{D_n}P^{-1}\longmapsto \tilde C_{D_n},
\qquad
X_{D_n}(\mathbf B_{\mathrm{even}})\longmapsto Y_{D_n},
\]
and
\[
T_{B_n}=P\tilde C_{B_n}P^{-1}\longmapsto \tilde C_{B_n},
\qquad
X_{B_n}(\mathbf B_{\mathrm{even}})\longmapsto Y_{B_n}.
\]
Therefore, we get
\[
\bigl\langle \tilde C_{D_n},Y_{D_n}\bigr\rangle_{\F}
=
\mathfrak{so}_{2n}(\F),
\qquad
\bigl\langle \tilde C_{B_n},Y_{B_n}\bigr\rangle_{\F}
=
\mathfrak{so}_{2n+1}(\F).
\]
The elements $\tilde C_{D_n},Y_{D_n}$ lie in
$\mathfrak{so}_{2n}(\F_q)$, and
$\tilde C_{B_n},Y_{B_n}$ lie in $\mathfrak{so}_{2n+1}(\F_q)$. By
Lemma~\ref{lem:descent}, the generation over $\F$ descends to $\F_q$.
This proves the three claims. This completes the proof of Proposition~\ref{prop:two-gen-even}.
\end{proof}

\subsection{Proof of Theorem~\ref{thm:two-generation classical}}
\label{subsec:proof}

We are now ready to prove the main theorem of this section, which establishes that all classical split simple Chevalley Lie algebras over $\mathbb{F}_q$ are two-generated. Here, we present the proof of Theorem~\ref{thm:two-generation classical}.
\begin{proof}[Proof of Theorem~\ref{thm:two-generation classical}]
Let $\mathfrak{g}\in\{\mathfrak{so}_{2n+1},\mathfrak{sp}_{2n},\mathfrak{so}_{2n}\}$ be one of the three Lie algebras, with corresponding form matrix $S$. We proceed by exhibiting an irreducible polynomial $f\in\mathbb{F}_q[x]$ of degree $n$ whose roots form a consistent set. Choose a normal element $\lambda\in\mathbb{F}_{q^n}$ over $\mathbb{F}_q$. It is well known that normal elements exist in any finite field (see~\cite{Hensel}). Let $f\in\mathbb{F}_q[x]$ be its minimal polynomial, and set
\[
t_i=\operatorname{Fr}^{i-1}(\lambda),\qquad i=1,\dots,n.
\]
Since $\alpha$ is normal, the conjugates $t_1,\dots,t_n$ form a basis of $\mathbb{F}_{q^n}$ over $\mathbb{F}_q$. Hence, as in Remark~\ref{rem:consistent}, the diagonal matrices $T$ appearing below are consistent.

\medskip
\noindent\textbf{Type $C_n$, $n\ge2$.}
Let $T=\operatorname{diag}(t_1,\dots,t_n,-t_1,\dots,-t_n)$, and let $X_{C_n}$ be as in Section~\ref{Section: Rationality of matrices}. By Lemma~\ref{lem:generation-splitting}, $T$ and $X_{C_n}$ generate $\mathfrak{sp}_{2n}(\mathbb{F}_{q^n})$. Let $\tilde C_{C_n}$ be the companion block matrix of Lemma~\ref{lem:companion-blocks}. By Lemma~\ref{lem:rational-Cn}, there is $P\in\operatorname{Sp}_{2n}(\mathbb{F}_{q^n})$ such that $P\tilde C_{C_n}P^{-1}=T$ and
\[
Y_{C_n}:=P^{-1}X_{C_n}P\in\mathfrak{sp}_{2n}(\mathbb{F}_q).
\]
Conjugation by $P^{-1}$ preserves $\mathfrak{sp}_{2n}(\mathbb{F}_{q^n})$ and sends $(T,X_{C_n})$ to $(\tilde C_{C_n},Y_{C_n})$. Hence, $\tilde C_{C_n}$ and $Y_{C_n}$ generate $\mathfrak{sp}_{2n}(\mathbb{F}_{q^n})$ over $\mathbb{F}_{q^n}$. Both lie in $\mathfrak{sp}_{2n}(\mathbb{F}_q)$. By Lemma~\ref{lem:descent}, they generate $\mathfrak{sp}_{2n}(\mathbb{F}_q)$ over $\mathbb{F}_q$. Thus, $\mathfrak{sp}_{2n}(\mathbb{F}_q)$ is $2$-generated.

\medskip
\noindent\textbf{Types $B_n$ and $D_n$, odd $n$.}
Assume $n\ge3$ is odd, and let $\mathbf B_{\mathrm{odd}}$ be as in Section~\ref{Section: Rationality of matrices}. The claim is exactly Proposition~\ref{prop:two-gen-odd}, which produces the required two $\mathbb{F}_q$-rational generators $\tilde C$ and $P^{-1}X(\mathbf B_{\mathrm{odd}})P$.

\medskip
\noindent\textbf{Types $B_n$ and $D_n$, even $n$.}
This is Proposition~\ref{prop:two-gen-even}. It covers $B_n$ ($B_2$ separately) and $D_n$ for $n\ge4$.

Finally, note that since $p\neq 2$,
\[
\mathfrak{so}_{3}(\F_q) \cong \mathfrak{sp}_{2}(\F_q) \cong \mathfrak{sl}_{2}(\F_q)=\mathfrak{psl}_{2}(\F_q).
\]
Thus, $\mathfrak{so}_{3}(\F_q)$ and $\mathfrak{sp}_{2}(\F_q)$ are also two-generated (see the main theorem~\cite{cjz}). This completes the proof of Theorem~\ref{thm:two-generation classical}.
\end{proof}

\begin{remark}
\label{rem:summary}
Theorem~\ref{thm:two-generation classical} covers the three classical families
$B_n$, $C_n$, and $D_n$ simultaneously; the case $A_n$ is treated in~\cite{cjz}. The
consistent torus is obtained from the minimal polynomial of a normal element
$\lambda\in\F_{q^n}$ over $\F_q$, which always exists
(see~\cite{Hensel}), and whose conjugates
$t_i=\Fr^{i-1}(\lambda)$ are linearly independent over $\F_q$. The cases $D_1$ and $D_2$
are excluded because $\mathfrak{so}_{2}$ and $\mathfrak{so}_{4}$ are not simple.
\end{remark}


\section{Exceptional type split simple Chevalley Lie algebras over $\mathbb{F}_q$}\label{section:exceptional}

In this section, we consider the split simple Chevalley Lie algebras of exceptional types: $\G_2(\F_q)$, $\mathfrak F_4(\F_q)$, $\mathfrak E_6(\F_q)$, $\mathfrak E_7(\F_q)$, and $\mathfrak E _8(\F_q)$. We refer the reader to an excellent book~\cite{Humphreys} for notations and definitions. We also refer to Section \ref{sub: Exceptional Type idea} for different notions and techniques used throughout this section. The main result that remains to be proven is the following.

\begin{theorem}[2-generation for exceptional Chevalley Lie algebras]\label{thm:two-generation exceptional}

The split simple Chevalley Lie algebras $\mathfrak g$ of type
\[\G_2(\F_q),\ \mathfrak F_4(\F_q),\ \mathfrak E_6(\F_q),\ \mathfrak E_7(\F_q),\ \text{and } \mathfrak E _8(\F_q)
\]
are $2$-generated, provided $(\mathfrak g, \Char(\F_q)) \ne (\G_2,2), (\G_2,3)$ and $(\mathfrak F_4,2)$.
\end{theorem}

We prove this theorem case by case starting from the Lie algebra of type $\mathfrak{G}_2(q)$.
The main working technique is as follows: Lemma \ref{lem:consistent} works for all split simple Chevalley Lie algebras over finite fields. In the case of exceptional Lie algebras, we first try to get a point which gives a non-zero evaluation at the polynomial $f_\mathfrak{g}$. Let $\hat{f_\mathfrak{g}}$ be the square-free part of $f_\mathfrak{g}$ such that $f_\mathfrak{g} = C_\mathfrak{g} \prod_{\{L_r \in \hat{f_\mathfrak{g}}\}} L_r^r$, where $L_r$'s are distinct linear factors in $\hat{f_\mathfrak{g}}$ with no integer divisors and have the first non-zero coefficient positive. Thus, $\hat{f_\mathfrak{g}}$ is uniquely defined for $f_\mathfrak{g}$. If $C_\mathfrak{g} = 0$ in $\F_q$, then $\mathfrak{g}$ has no consistent element. So, we first aim to find characteristics over which $C_\mathfrak{g} \ne 0$ and, in that case, by Theorem \ref{Alon's theorem}, we aim to get a point $s$ for large enough $q$ such that $\hat{f_\mathfrak{g}}(s) \ne 0$ (if it exists). Consequently, ${f_\mathfrak{g}}(s) \ne 0$ and hence we get a consistent element. In a later section, we show that by Lemma \ref{lem:generation-splitting}, $\mathfrak{g}$ is $2$-generated for these large enough $q$ values. For small values of $q$, we check $2$-generation by using a $\textsc {GAP}$ \cite{Gap} program.

\subsection{$2$-generation for $\mathfrak{G}_2(q)$}

Let \(\Phi\) be the root system of type \(\G_2\) with simple roots \(\alpha_1\) (short) and \(\alpha_2\) (long), and Cartan matrix
\[
A=\bigl(A_{ij}\bigr)_{1\le i,j\le 2},
\qquad
A_{ij}=\alpha_i(h_j),
\qquad
A=\begin{pmatrix} 2 & -1\\ -3 & 2\end{pmatrix}.
\]
Fix the basis \(\{h_1,h_2\}\) of \(\mathfrak h\) dual to the simple roots, and let \(x_1,x_2\) be the dual coordinates on \(\mathfrak h\), so that a point \(T=s_1h_1+s_2h_2\) is identified with \((s_1,s_2)\in\mathbb F_q^2\) and \(x_i(T)=s_i\) for \(i=1,2\). The set of positive roots is
\[
\Phi^+=\{\alpha_1,\ \alpha_2,\ \alpha_1+\alpha_2,\ 
2\alpha_1+\alpha_2,\ 3\alpha_1+\alpha_2,\ 
3\alpha_1+2\alpha_2\}.
\]
The six positive roots correspond to the following linear factors in \((x_1,x_2)\):
\begin{align*}
&\alpha_1(x) = 2x_1-x_2, &
&\alpha_2(x) = -3x_1+2x_2,\\
&(\alpha_1+\alpha_2)(x) = -x_1+x_2, &
&(2\alpha_1+\alpha_2)(x) = x_1,\\
&(3\alpha_1+\alpha_2)(x) = 3x_1-x_2, &
&(3\alpha_1+2\alpha_2)(x) = x_2.
\end{align*}

\begin{lemma}\label{lemma: consistent for G_2}
The Lie algebra \(\G_2(\mathbb{F}_q)\) contains a consistent element if \(q>11\) and \(\operatorname{Char}(\mathbb{F}_q) \ne 2,3\). 
\end{lemma}

\begin{proof}
We input the Cartan matrix of $\G_2$ in a $\textsc {GAP}$ code similar to the one listed in Listing \ref{lst:gap-for-F4-linear-factors} to find distinct linear factors in \(\hat{f}_{\G_2}\), the scalar $C_{\G_2}$ and its prime factorization. The $\textsc {GAP}$ output shows that there are $12$ distinct linear factors in \(\hat{f}_{\G_2}\), written as
\begin{align*}
=&2x_1-x_2, \ 3x_1-2x_2,\ x_1-x_2,\ x_1,\ 3x_1-x_2, \ x_2\\
&5x_1-3x_2, \ 5x_1-2x_2,\ 4x_1-3x_2,\ x_1-2x_2,\ 4x_1-x_2, \ x_1+x_2.
\end{align*}
The prime divisors of $C_{\G_2}$ as obtained by the $\textsc {GAP}$ are $2$ and $3$. So, $\G_2(\F_q)$ has no consistent element if $\Char(\F_q) = 2,3$. The coefficient of $x_1x_2^{11}$ in \(\hat{f}_{\G_2}\) is $-2^33^2$ which is non-zero if $\Char(\F_q) \ne 2,3$. Hence, by Theorem \ref{Alon's theorem} if $q>11$, $\Char(\F_q) \ne 2,3$ and $S_1=S_2=\mathbb{F}_q$, then there exists $(s_1,s_2)\in\mathbb{F}_q^2$ with \(\hat{f}_{\G_2}(s_1,s_2)\neq 0\). Thus, for this tuple \(f_{\G_2}(s_1, s_2) \ne 0\), and hence $\G_2(\F_q)$ contains a consistent element if $q >11$ and $\Char(\F_q) \ne 2,3$.
\end{proof}

\begin{proposition}\label{prop: 2-generation for G_2}
The Lie algebra \(\G_2(\mathbb{F}_q)\) is \(2\)-generated if \(\Char(\mathbb{F}_q) \ne 2,3\).
\end{proposition}

\begin{proof}
From Lemma \ref{lemma: consistent for G_2} and Lemma \ref{lem:generation-splitting}, \(\G_2(\mathbb{F}_q)\) is \(2\)-generated if \(q>11\) and \(\Char(\mathbb{F}_q) \ne 2,3\). For \(q=5,7\) and \(11\), we use the \textsc{GAP} code in Listing \ref{lst:gap 2-generation of small q} to verify that \(\G_2(\mathbb{F}_q)\) is \(2\)-generated.
\end{proof}


\subsection{$2$-generation for $\mathfrak F_4(\F_q)$}

Let \(\Phi\) be the root system of type \(\mathfrak F_4\) having simple roots \(\alpha_1,\alpha_2\) as long roots and \(\alpha_3,\alpha_4\) as short roots with the following Cartan matrix
\[
A=\bigl(A_{ij}\bigr)_{1\le i,j\le 4},
\qquad
A_{ij}=\alpha_i(h_j),
\qquad
A=
\begin{pmatrix}
2 & -1 & 0 & 0\\
-1 & 2 & -2 & 0\\
0 & -1 & 2 & -1\\
0 & 0 & -1 & 2
\end{pmatrix}.
\]
Fix the basis \(\{h_1,h_2,h_3,h_4\}\) of \(\mathfrak h\) dual to the simple roots, and let \(x_1,x_2,x_3,x_4\) be the dual coordinates on \(\mathfrak h\), so that a point
\(T=s_1h_1+s_2h_2+s_3h_3+s_4h_4\)
is identified with \((s_1,s_2,s_3,s_4)\in\mathbb F_q^4\) and \(x_i(T)=s_i\) for \(i=1,2,3,4\).
Consequently, the linear factors for simple roots are
\[
\begin{aligned}
\alpha_1(x)&=2x_1-x_2, \qquad \qquad \qquad \alpha_2(x)=-x_1+2x_2-2x_3,\\
\alpha_3(x)&=-x_2+2x_3-x_4,  \quad \quad \quad  \alpha_4(x)=-x_3+2x_4.
\end{aligned}
\]

\noindent For a positive root \(
\beta=c_1\alpha_1+c_2\alpha_2+c_3\alpha_3+c_4\alpha_4,\)
the associated linear factor is
\[
\beta(x)=d_1x_1+d_2x_2+d_3x_3+d_4x_4,
\qquad \text{where }
(d_1,d_2,d_3,d_4)=(c_1,c_2,c_3,c_4)A.
\]

We input the Cartan matrix for $\mathfrak{F}_4$ in the \textsc{GAP} code listed in listing \ref{lst:gap-for-F4-linear-factors} to generate the positive roots of \(\mathfrak{F}_4\) by applying the simple reflections and then compute their associated linear factors. Since \(\mathfrak{F}_4\) has \(24\) positive roots, the resulting table contains \(24\) tuples $(c_1, c_2,c_3,c_4)$ for positive roots of the form $\sum_{i=1}^4c_i\alpha_i$ and the coordinates $(d_1,d_2,d_3,d_4)$ of their associated linear factors $\sum_{i=1}^4d_ix_i$. We subsequently use the \(24\) linear factors associated with positive roots to construct all pairwise differences and pairwise sums with distinct indices.

\begin{table}[ht]
\centering
\captionsetup{font=tiny}
\caption{Positive roots of $\mathfrak{F}_4$ and their corresponding linear factor coefficients.}
\tiny
\renewcommand{\arraystretch}{1.15}
\setlength{\tabcolsep}{5pt}
\begin{tabular}{c c|c c}
\hline
Root coefficients & Linear factor coefficients &
Root coefficients & Linear factor coefficients \\
\hline
$(1,0,0,0)$ & $(2,-1,0,0)$ &
$(1,2,2,0)$ & $(0,1,0,-2)$ \\

$(0,1,0,0)$ & $(-1,2,-2,0)$ &
$(1,1,2,2)$ & $(1,-1,0,2)$ \\

$(0,0,1,0)$ & $(0,-1,2,-1)$ &
$(1,1,1,1)$ & $(1,0,-1,1)$ \\

$(0,0,0,1)$ & $(0,0,-1,2)$ &
$(0,1,2,1)$ & $(-1,0,1,0)$ \\

$(1,1,0,0)$ & $(1,1,-2,0)$ &
$(1,2,2,2)$ & $(0,1,-2,2)$ \\

$(0,1,2,0)$ & $(-1,0,2,-2)$ &
$(1,1,2,1)$ & $(1,-1,1,0)$ \\

$(0,1,1,0)$ & $(-1,1,0,-1)$ &
$(1,2,4,2)$ & $(0,-1,2,0)$ \\

$(0,0,1,1)$ & $(0,-1,1,1)$ &
$(1,2,2,1)$ & $(0,1,-1,0)$ \\

$(1,1,2,0)$ & $(1,-1,2,-2)$ &
$(1,3,4,2)$ & $(-1,1,0,0)$ \\

$(0,1,2,2)$ & $(-1,0,0,2)$ &
$(1,2,3,1)$ & $(0,0,1,-1)$ \\

$(1,1,1,0)$ & $(1,0,0,-1)$ &
$(2,3,4,2)$ & $(1,0,0,0)$ \\

$(0,1,1,1)$ & $(-1,1,-1,1)$ &
$(1,2,3,2)$ & $(0,0,0,1)$ \\
\hline
\end{tabular}
\label{tab:F4-root-linear-forms}
\end{table}


\begin{lemma}\label{lemma (1): consistent for F_4}
The Lie algebra $\mathfrak F_4(\mathbb{F}_q)$ contains a consistent element if $q>155$ and $\Char (\F_q) \ne 2,3$. 
\end{lemma}

\begin{proof}
A direct computation using the \textsc{GAP} code in Listing \ref{lst:gap-for-F4-linear-factors} shows that the prime divisor of $C_{\mathfrak F_4}$ is $2$. So, $\mathfrak F_4(\F_q)$ has no consistent element if $\Char(\F_q) = 2$. Further, as per the \textsc{GAP} output, the polynomial $\hat{f}_{\mathfrak{F}_4}$ contains $192$ linear factors, out of which $155$ contain a non-zero coefficient for $x_1$, $31$ factors have a zero coefficient for $x_1$ but a non-zero coefficient for $x_2$, and $5$ factors have zero coefficients for both $x_1$ and $x_2$ with a non-zero coefficient for $x_3$.

Consequently, the coefficient of $x_1^{155} x_2^{31} x_3^{5} x_4^1$ in the product of all these linear factors is obtained by multiplying three parts together:

\begin{itemize}
\item The $x_1$-coefficients of the $155$ factors in which they are non-zero;
\item The $x_2$-coefficients of the $31$ factors in which the $x_1$-coefficient is zero and the $x_2$-coefficient is non-zero;
\item The $x_3$-coefficients of the $5$ factors in which the coefficients of both $x_1$ and $x_2$ are zero and the $x_3$-coefficient is non-zero; and
\item The $x_4$-coefficient of the remaining factor.
\end{itemize} 

Using the above idea, the coefficient of $x_1^{155} x_2^{31} x_3^{5} x_4^1$ is calculated through the same $\textsc {GAP}$ code. This coefficient is non-zero if $\Char(\F_q) \ne 2,3$. Hence, by Theorem \ref{Alon's theorem}, if $q>155$, $\Char(\F_q) \ne 2,3$ and $S_1=S_2=S_3=S_4=\mathbb{F}_q$, then there exists $(s_1,s_2,s_3,s_4)\in\mathbb{F}_q^4$ with \(\hat{f}_{\mathfrak F_4}(s_1,s_2,s_3,s_4)\neq 0\), and thus \(f_{\mathfrak F_4}(s_1, s_2,s_3,s_4) \ne 0\). Consequently, $\mathfrak F_4(\F_q)$ contains a consistent element if $q >155$ and $\Char(\F_q) \ne 2,3$. 
\end{proof}

\begin{lemma}\label{lemma: consistent for F_4 (char=3)}
The Lie algebra $\mathfrak F_4(\mathbb{F}_q)$ contains a consistent element if $q>27$ and $\Char (\F_q) = 3$. 
\end{lemma}

\begin{proof}
When $\Char (\F_q) = 3$, the $\textsc {GAP}$ code further normalizes the linear factors and checks for distinct factors. The output shows that the product of scalars that comes out of this normalization step is $\equiv 2 \mod{(3)}$. Also, after normalization, we are left with only $40$ distinct linear factors. We compute the coefficient of $x_1^{27} x_2^{9} x_3^{3} x_4^1$ in the product of these distinct linear factors modulo $3$. This coefficient is non-zero and hence, by applying Theorem \ref{Alon's theorem}, if $q>27$, $\Char(\F_q) =3$ and $S_1=S_2=S_3=S_4=\mathbb{F}_q$, then there exists a consistent element in $\mathfrak F_4(\F_q)$. 
\end{proof}

\begin{proposition}\label{prop: 2-generation for F_4}
The Lie algebra $\mathfrak F_4(\mathbb{F}_q)$ is $2$-generated if $\Char(\F_q) \ne 2$. 
\end{proposition}

\begin{proof}
Let $\Char(\F_q) \ne 2$. By Lemmas \ref{lemma (1): consistent for F_4} and \ref{lemma: consistent for F_4 (char=3)}, $\mathfrak F_4(\mathbb{F}_q)$ contains a consistent element if $q > 155$. Thus, by Theorem \ref{Alon's theorem}, $\mathfrak F_4(\mathbb{F}_q)$ is $2$-generated in this case. For $q \leq 155$, we run a similar \textsc{GAP} code as in Listing \ref{lst:gap 2-generation of small q} for $\mathfrak F_4(\mathbb{F}_q)$ by defining $x$ and $y$ as follows.  

\begin{align*}
 & x := b[1][1] + b[1][2] + b[1][3] + b[1][4] + b[3][1],\\
  &y := b[2][1] + b[2][2] + b[2][3] + b[2][4] + b[3][2].   
\end{align*}

The \textsc{GAP} output verifies that $\mathfrak F_4(\mathbb{F}_q)$ is $2$-generated for $q \leq 155$ as well.\\
\end{proof}




\subsection{$2$-generation for $\mathfrak E_6(\F_q))$}

Let \(\Phi\) be the root system of type \(\mathfrak E_6(\F_q)\), with simple roots \(\alpha_1,\ldots,\alpha_6\).
The Cartan matrix is $A=(A_{ij})_{1\le i,j\le 6}, \ A_{ij}=\alpha_i(h_j),
$ where
\[ A= \begin{pmatrix}
2&0&-1&0&0&0\\ 0&2&0&-1&0&0\\ -1&0&2&-1&0&0\\
0&-1&-1&2&-1&0\\ 0&0&0&-1&2&-1\\ 0&0&0&0&-1&2
\end{pmatrix}.\]
Fix the basis \(\{h_1,\ldots,h_6\}\) of \(\mathfrak h\) dual to the simple roots, and let \(x_1,\ldots,x_6\) be the dual coordinates on \(\mathfrak h\).
Consequently, the linear factors corresponding to the simple roots are
\[
\begin{aligned}
\alpha_1(x)&=2x_1-x_3, \qquad \qquad \quad \quad \alpha_2(x)=2x_2-x_4,\\
\alpha_3(x)&=-x_1+2x_3-x_4, \quad \qquad \alpha_4(x)=-x_2-x_3+2x_4-x_5,\\
\alpha_5(x)&=-x_4+2x_5-x_6, \quad \qquad\alpha_6(x)=-x_5+2x_6.
\end{aligned}
\]
For a positive root
$\beta=\sum_{i=1}^6 c_i\alpha_i,
$
the associated linear factor is
$\beta(x)=\sum_{i=1}^6 d_i x_i,$ where $(d_1,\ldots,d_6)=(c_1,\ldots,c_6)A.$ We use \textsc{GAP} to generate the positive roots of \(\mathfrak E_6(\F_q)\) and then compute their associated linear factors (see Table \ref{tab:E6-root-linear-forms}). 

\begin{table}[ht]
\centering
\captionsetup{font=tiny}
\caption{Positive roots of $\mathfrak E_6(\F_q)$ and their corresponding
linear factor coefficients.}
\tiny
\renewcommand{\arraystretch}{1.15}
\setlength{\tabcolsep}{5pt}
\begin{tabular}{c c|c c}
\hline
\text{Root coefficients} & \text{Linear factor coefficients}
& \text{Root coefficients} & \text{Linear factor coefficients} \\ \hline
(1,0,0,0,0,0) & (2,0,-1,0,0,0) &
(0,1,1,1,1,0) & (-1,1,1,-1,1,-1) \\

(0,1,0,0,0,0) & (0,2,0,-1,0,0) &
(0,1,0,1,1,1) & (0,1,-1,0,0,1) \\

(0,0,1,0,0,0) & (-1,0,2,-1,0,0) &
(0,0,1,1,1,1) & (-1,-1,1,0,0,1) \\

(0,0,0,1,0,0) & (0,-1,-1,2,-1,0) &
(1,1,1,1,1,0) & (1,1,0,-1,1,-1) \\

(0,0,0,0,1,0) & (0,0,0,-1,2,-1) &
(1,0,1,1,1,1) & (1,-1,0,0,0,1) \\

(0,0,0,0,0,1) & (0,0,0,0,-1,2) &
(0,1,1,2,1,0) & (-1,0,0,1,0,-1) \\

(1,0,1,0,0,0) & (1,0,1,-1,0,0) &
(0,1,1,1,1,1) & (-1,1,1,-1,0,1) \\

(0,1,0,1,0,0) & (0,1,-1,1,-1,0) &
(1,1,1,2,1,0) & (1,0,-1,1,0,-1) \\

(0,0,1,1,0,0) & (-1,-1,1,1,-1,0) &
(1,1,1,1,1,1) & (1,1,0,-1,0,1) \\

(0,0,0,1,1,0) & (0,-1,-1,1,1,-1) &
(0,1,1,2,1,1) & (-1,0,0,1,-1,1) \\

(0,0,0,0,1,1) & (0,0,0,-1,1,1) &
(1,1,2,2,1,0) & (0,0,1,0,0,-1) \\

(1,0,1,1,0,0) & (1,-1,0,1,-1,0) &
(1,1,1,2,1,1) & (1,0,-1,1,-1,1) \\

(0,1,1,1,0,0) & (-1,1,1,0,-1,0) &
(0,1,1,2,2,1) & (-1,0,0,0,1,0) \\

(0,1,0,1,1,0) & (0,1,-1,0,1,-1) &
(1,1,2,2,1,1) & (0,0,1,0,-1,1) \\

(0,0,1,1,1,0) & (-1,-1,1,0,1,-1) &
(1,1,1,2,2,1) & (1,0,-1,0,1,0) \\

(0,0,0,1,1,1) & (0,-1,-1,1,0,1) &
(1,1,2,2,2,1) & (0,0,1,-1,1,0) \\

(1,1,1,1,0,0) & (1,1,0,0,-1,0) &
(1,1,2,3,2,1) & (0,-1,0,1,0,0) \\

(1,0,1,1,1,0) & (1,-1,0,0,1,-1) &
(1,2,2,3,2,1) & (0,1,0,0,0,0)\\
\hline
\end{tabular}
\label{tab:E6-root-linear-forms}
\end{table}

\begin{lemma}\label{lemma (1): consistent for E_6}
The Lie algebra $\mathfrak E_6(\mathbb{F}_q)$ contains a consistent element if $q>401$ and $\Char (\F_q) \ne 2,3$. 
\end{lemma}

\begin{proof}
A direct computation using the \textsc{GAP} code similar to Listing \ref{lst:gap-for-F4-linear-factors} shows that $C_{\mathfrak E_6}$ has no prime divisor. Further, the polynomial $\hat{f}_{\mathfrak{E}_6}$ contains $531$ linear factors. Let $d_i$ be the number of linear factors in $\hat{f}_{\mathfrak{E}_6}$ having zero coefficient of $x_1, x_2, \dots, x_{i-1}$ and a non-zero coefficient of $x_i$, then as per the \textsc{GAP} output, $d_1 = 401$, $d_2=95$, $d_3=28$, $d_4=6$, $d_5=1$ and $d_6=0$.
Thus, we calculate the coefficient of $x_1^{401} x_2^{95} x_3^{28} x_4^6 x_5$ in the product of all linear factors in $\hat{f}_{\mathfrak{E}_6}$ through the same $\textsc {GAP}$ code. This coefficient is non-zero if $\Char(\F_q) \ne 2,3$. Hence, by Theorem \ref{Alon's theorem}, if $q>401$ and $\Char(\F_q) \ne 2,3$, then $\mathfrak E_6(\F_q)$ contains a consistent element. 
\end{proof}

\begin{lemma}\label{lemma: consistent for E_6 (char=2)}
The Lie algebra $\mathfrak E_6(\mathbb{F}_q)$ contains a consistent element if $q>32$ and $\Char (\F_q) = 2$. 
\end{lemma}

\begin{proof}
When $\Char (\F_q) = 2$, the $\textsc {GAP}$ code further normalizes the linear factors and checks for distinct factors. The product of scalars that comes out of this normalization step is $\equiv 1 \mod{(2)}$. Also, after normalization, we are left with only $63$ distinct linear factors. We compute the coefficient of $x_1^{32} x_2^{16} x_3^{8} x_4^4 x_5^2 x_6$ in the product of these distinct linear factors modulo $2$. This coefficient is non-zero and hence, by Theorem \ref{Alon's theorem}, if $q>32$ and $\Char(\F_q) =2$, then there exists a consistent element in $\mathfrak E_6(\F_q)$. 
\end{proof}

\begin{lemma}\label{lemma: consistent for E_6 (char=3)}
The Lie algebra $\mathfrak E_6(\mathbb{F}_q)$ contains a consistent element if $q>81$ and $\Char (\F_q) = 3$. 
\end{lemma}

\begin{proof}
When $\Char (\F_q) = 3$, further normalization involves scalars whose product is $\equiv 2 \mod{(3)}$ and we have only $121$ distinct linear factors in this case. We compute the coefficient of $x_1^{81} x_2^{27} x_3^{9} x_4^3 x_5$ in the product of these distinct linear factors modulo $3$. This coefficient is non-zero and hence, by Theorem \ref{Alon's theorem}, there exists a consistent element in this case. 
\end{proof}

\begin{proposition}\label{prop: 2-generation for E_6}
The Lie algebra $\mathfrak E_6(\mathbb{F}_q)$ is $2$-generated. 
\end{proposition}

\begin{proof}
By Lemmas \ref{lemma (1): consistent for E_6}, \ref{lemma: consistent for E_6 (char=2)} and \ref{lemma: consistent for E_6 (char=3)}, $\mathfrak E_6(\mathbb{F}_q)$ contains a consistent element if $q > 401$, $\Char(\F_q) \ne 2,3$; $q > 32$, $\Char(\F_q) = 2$; and $q > 81$, $\Char(\F_q) = 3$. Thus, by Lemma \ref{lem:generation-splitting}, $\mathfrak E_6(\mathbb{F}_q)$ is $2$-generated in this case. For the remaining values of $q$, a similar \textsc{GAP} code as in Listing \ref{lst:gap 2-generation of small q} verifies that $\mathfrak E_6(\mathbb{F}_q)$ is generated by the following elements.

\begin{align*}
 & x := b[1][1] + b[1][2] + b[1][3] + b[1][4]  + b[1][5] + b[1][6] + b[3][1],\\
  &y := b[2][1] + b[2][2] + b[2][3] + b[2][4] + b[2][5] + b[2][6] + b[3][2].   
\end{align*}
\end{proof}

\subsection{$2$-generation for $\mathfrak E_7(\F_q)$}

Let \(\Phi\) be the root system of type \(\mathfrak E_7(\F_q)\), with simple roots \(\alpha_1,\ldots,\alpha_7\). The Cartan matrix is \(A=(A_{ij})_{1\le i,j\le 7}\), \(A_{ij}=\alpha_i(h_j)\), where
\[
A=
\begin{pmatrix}
2&0&-1&0&0&0&0\\
0&2&0&-1&0&0&0\\
-1&0&2&-1&0&0&0\\
0&-1&-1&2&-1&0&0\\
0&0&0&-1&2&-1&0\\
0&0&0&0&-1&2&-1\\
0&0&0&0&0&-1&2
\end{pmatrix}.
\]
Fix the basis \(\{h_1,\ldots,h_7\}\) of \(\mathfrak h\) dual to the simple roots, and let \(x_1,\ldots,x_7\) be the dual coordinates on \(\mathfrak h\). Consequently, the linear factors corresponding to the simple roots are
\[
\begin{aligned}
\alpha_1(x)&=2x_1-x_3, \qquad \qquad \qquad
\alpha_2(x)=2x_2-x_4,\\
\alpha_3(x)&=-x_1+2x_3-x_4, \qquad \quad
\alpha_4(x)=-x_2-x_3+2x_4-x_5,\\
\alpha_5(x)&=-x_4+2x_5-x_6, \qquad \quad
\alpha_6(x)=-x_5+2x_6-x_7,\\
\alpha_7(x)&=-x_6+2x_7.
\end{aligned}
\]

We again use \textsc{GAP} to generate the positive roots of \(\mathfrak E_7(\F_q)\) and then compute their associated linear factors (see Table \ref{tab:E7-root-linear-forms}).

\begin{tiny}
\begin{longtable}{c c c c}
\captionsetup{font=tiny}
\caption{Positive roots and their linear factors for $\mathfrak E_7(\F_q)$.}
\label{tab:E7-root-linear-forms}\\
\hline
Positive root coefficients & Linear factor coefficients &
Positive root coefficients & Linear factor coefficients \\
\hline
\endfirsthead

\hline
Positive root coefficients & Linear factor coefficients &
Positive root coefficients & Linear factor coefficients \\
\hline
\endhead
$(1,0,0,0,0,0,0)$ & $(2,0,-1,0,0,0,0)$ & $(1,1,1,1,1,1,0)$ & $(1,1,0,-1,0,1,-1)$ \\
$(0,1,0,0,0,0,0)$ & $(0,2,0,-1,0,0,0)$ & $(1,0,1,1,1,1,1)$ & $(1,-1,0,0,0,0,1)$ \\
$(0,0,1,0,0,0,0)$ & $(-1,0,2,-1,0,0,0)$ & $(0,1,1,2,1,1,0)$ & $(-1,0,0,1,-1,1,-1)$ \\
$(0,0,0,1,0,0,0)$ & $(0,-1,-1,2,-1,0,0)$ & $(0,1,1,1,1,1,1)$ & $(-1,1,1,-1,0,0,1)$ \\
$(0,0,0,0,1,0,0)$ & $(0,0,0,-1,2,-1,0)$ & $(1,1,2,2,1,0,0)$ & $(0,0,1,0,0,-1,0)$ \\
$(0,0,0,0,0,1,0)$ & $(0,0,0,0,-1,2,-1)$ & $(1,1,1,2,1,1,0)$ & $(1,0,-1,1,-1,1,-1)$ \\
$(0,0,0,0,0,0,1)$ & $(0,0,0,0,0,-1,2)$ & $(1,1,1,1,1,1,1)$ & $(1,1,0,-1,0,0,1)$ \\
$(1,0,1,0,0,0,0)$ & $(1,0,1,-1,0,0,0)$ & $(0,1,1,2,2,1,0)$ & $(-1,0,0,0,1,0,-1)$ \\
$(0,1,0,1,0,0,0)$ & $(0,1,-1,1,-1,0,0)$ & $(0,1,1,2,1,1,1)$ & $(-1,0,0,1,-1,0,1)$ \\
$(0,0,1,1,0,0,0)$ & $(-1,-1,1,1,-1,0,0)$ & $(1,1,2,2,1,1,0)$ & $(0,0,1,0,-1,1,-1)$ \\
$(0,0,0,1,1,0,0)$ & $(0,-1,-1,1,1,-1,0)$ & $(1,1,1,2,2,1,0)$ & $(1,0,-1,0,1,0,-1)$ \\
$(0,0,0,0,1,1,0)$ & $(0,0,0,-1,1,1,-1)$ & $(1,1,1,2,1,1,1)$ & $(1,0,-1,1,-1,0,1)$ \\
$(0,0,0,0,0,1,1)$ & $(0,0,0,0,-1,1,1)$ & $(0,1,1,2,2,1,1)$ & $(-1,0,0,0,1,-1,1)$ \\
$(1,0,1,1,0,0,0)$ & $(1,-1,0,1,-1,0,0)$ & $(1,1,2,2,2,1,0)$ & $(0,0,1,-1,1,0,-1)$ \\
$(0,1,1,1,0,0,0)$ & $(-1,1,1,0,-1,0,0)$ & $(1,1,2,2,1,1,1)$ & $(0,0,1,0,-1,0,1)$ \\
$(0,1,0,1,1,0,0)$ & $(0,1,-1,0,1,-1,0)$ & $(1,1,1,2,2,1,1)$ & $(1,0,-1,0,1,-1,1)$ \\
$(0,0,1,1,1,0,0)$ & $(-1,-1,1,0,1,-1,0)$ & $(0,1,1,2,2,2,1)$ & $(-1,0,0,0,0,1,0)$ \\
$(0,0,0,1,1,1,0)$ & $(0,-1,-1,1,0,1,-1)$ & $(1,1,2,3,2,1,0)$ & $(0,-1,0,1,0,0,-1)$ \\
$(0,0,0,0,1,1,1)$ & $(0,0,0,-1,1,0,1)$ & $(1,1,2,2,2,1,1)$ & $(0,0,1,-1,1,-1,1)$ \\
$(1,1,1,1,0,0,0)$ & $(1,1,0,0,-1,0,0)$ & $(1,1,1,2,2,2,1)$ & $(1,0,-1,0,0,1,0)$ \\
$(1,0,1,1,1,0,0)$ & $(1,-1,0,0,1,-1,0)$ & $(1,2,2,3,2,1,0)$ & $(0,1,0,0,0,0,-1)$ \\
$(0,1,1,1,1,0,0)$ & $(-1,1,1,-1,1,-1,0)$ & $(1,1,2,3,2,1,1)$ & $(0,-1,0,1,0,-1,1)$ \\
$(0,1,0,1,1,1,0)$ & $(0,1,-1,0,0,1,-1)$ & $(1,1,2,2,2,2,1)$ & $(0,0,1,-1,0,1,0)$ \\
$(0,0,1,1,1,1,0)$ & $(-1,-1,1,0,0,1,-1)$ & $(1,2,2,3,2,1,1)$ & $(0,1,0,0,0,-1,1)$ \\
$(0,0,0,1,1,1,1)$ & $(0,-1,-1,1,0,0,1)$ & $(1,1,2,3,2,2,1)$ & $(0,-1,0,1,-1,1,0)$ \\
$(1,1,1,1,1,0,0)$ & $(1,1,0,-1,1,-1,0)$ & $(1,2,2,3,2,2,1)$ & $(0,1,0,0,-1,1,0)$ \\
$(1,0,1,1,1,1,0)$ & $(1,-1,0,0,0,1,-1)$ & $(1,1,2,3,3,2,1)$ & $(0,-1,0,0,1,0,0)$ \\
$(0,1,1,2,1,0,0)$ & $(-1,0,0,1,0,-1,0)$ & $(1,2,2,3,3,2,1)$ & $(0,1,0,-1,1,0,0)$ \\
$(0,1,1,1,1,1,0)$ & $(-1,1,1,-1,0,1,-1)$ & $(1,2,2,4,3,2,1)$ & $(0,0,-1,1,0,0,0)$ \\
$(0,1,0,1,1,1,1)$ & $(0,1,-1,0,0,0,1)$ & $(1,2,3,4,3,2,1)$ & $(-1,0,1,0,0,0,0)$ \\
$(0,0,1,1,1,1,1)$ & $(-1,-1,1,0,0,0,1)$ & $(2,2,3,4,3,2,1)$ & $(1,0,0,0,0,0,0)$ \\
$(1,1,1,2,1,0,0)$ & $(1,0,-1,1,0,-1,0)$ &  \\
\hline
\end{longtable}
\end{tiny}

\begin{lemma}\label{lemma (1): consistent for E_7}
The Lie algebra $\mathfrak E_7(\mathbb{F}_q)$ contains a consistent element if $q>1037$ and $\Char (\F_q) \ne 2,3$. 
\end{lemma}

\begin{proof}
The \textsc{GAP} code similar to Listing \ref{lst:gap-for-F4-linear-factors} shows that $C_{\mathfrak E_7}$ has no prime divisor and the polynomial $\hat{f}_{\mathfrak{E}_7}$ contains $1449$ linear factors. Let $d_i$ be the number of linear factors in $\hat{f}_{\mathfrak{E}_7}$ having zero coefficient of $x_1, x_2, \dots, x_{i-1}$ and a non-zero coefficient of $x_i$, then as per the \textsc{GAP} output for ${\mathfrak E_7}(\F_q)$, $d_1 = 1037$, $d_2=299$, $d_3=82$, $d_4=22$, $d_5=7$, $d_6=2$ and $d_7=0$.
Thus, we calculate the coefficient of $x_1^{1037} x_2^{299} x_3^{82} x_4^{22} x_5^7 x_6^2$ in the product of all linear factors in $\hat{f}_{\mathfrak{E}_7}$ through the same $\textsc {GAP}$ code. This coefficient is non-zero if $\Char(\F_q) \ne 2,3$. Hence, by Theorem \ref{Alon's theorem}, if $q>1037$ and $\Char(\F_q) \ne 2,3$, then $\mathfrak E_7(\F_q)$ contains a consistent element. 
\end{proof}

\begin{lemma}\label{lemma: consistent for E_7 (char=2)}
The Lie algebra $\mathfrak E_7(\mathbb{F}_q)$ contains a consistent element if $q>32$ and $\Char (\F_q) = 2$. 
\end{lemma}

\begin{proof}
When $\Char (\F_q) = 2$, the $\textsc {GAP}$ output for $\mathfrak E_7(\F_q)$ matches exactly with the output of $\mathfrak E_6(\F_q)$. Hence, this lemma follows similar to Lemma \ref{lemma: consistent for E_6 (char=2)} 
\end{proof}

\begin{lemma}\label{lemma: consistent for E_7 (char=3)}
The Lie algebra $\mathfrak E_7(\mathbb{F}_q)$ contains a consistent element if $q>525$ and $\Char (\F_q) = 3$. 
\end{lemma}

\begin{proof}
When $\Char (\F_q) = 3$, further normalization involves scalars whose product is $\equiv 2 \mod{(3)}$ and we have only $777$ distinct linear factors in this case. We compute the coefficient of $x_1^{525} x_2^{171} x_3^{57} x_4^{17} x_5^5 x_6^2$ in the product of these distinct linear factors modulo $3$. This coefficient is non-zero and hence, by Theorem \ref{Alon's theorem}, there exists a consistent element in this case. 
\end{proof}

\begin{proposition}\label{prop: 2-generation for E_7}
The Lie algebra $\mathfrak E_7(\mathbb{F}_q)$ is $2$-generated.
\end{proposition}

\begin{proof}
By Lemmas \ref{lemma (1): consistent for E_7}, \ref{lemma: consistent for E_7 (char=2)} and \ref{lemma: consistent for E_7 (char=3)}, $\mathfrak E_7(\mathbb{F}_q)$ contains a consistent element if $q > 1037$, $\Char(\F_q) \ne 2,3$; $q > 32$, $\Char(\F_q) = 2$; and $q > 777$, $\Char(\F_q) = 3$. Thus, by Lemma \ref{lem:generation-splitting}, $\mathfrak E_7(\mathbb{F}_q)$ is $2$-generated in this case. For the remaining values of $q$, a similar \textsc{GAP} code as in Listing \ref{lst:gap 2-generation of small q} verifies that $\mathfrak E_7(\mathbb{F}_q)$ is generated by the following elements.

\begin{align*}
 & x := b[1][1] + b[1][2] + b[1][3] + b[1][4]  + b[1][5] + b[1][6] + b[1][7] + b[3][1],\\
  &y := b[2][1] + b[2][2] + b[2][3] + b[2][4] + b[2][5] + b[2][6] + b[2][7] + b[3][2].   
\end{align*}
\end{proof}

\subsection{$2$-generation for $\mathfrak E_8(\F_q)$}

Let \(\Phi\) be the root system of type \(\mathfrak E_8(\F_q)\), with simple roots
\(\alpha_1,\ldots,\alpha_8\). The Cartan matrix is
\(A=(A_{ij})_{1\le i,j\le 8}\), \(A_{ij}=\alpha_i(h_j)\), where
\[
A=
\begin{pmatrix}
2&0&-1&0&0&0&0&0\\
0&2&0&-1&0&0&0&0\\
-1&0&2&-1&0&0&0&0\\
0&-1&-1&2&-1&0&0&0\\
0&0&0&-1&2&-1&0&0\\
0&0&0&0&-1&2&-1&0\\
0&0&0&0&0&-1&2&-1\\
0&0&0&0&0&0&-1&2
\end{pmatrix}.
\]

The linear factors corresponding to the simple roots are
\[
\begin{aligned}
\alpha_1(x)&=2x_1-x_3, & \qquad 
\alpha_2(x)&=2x_2-x_4,\\
\alpha_3(x)&=-x_1+2x_3-x_4, &
\alpha_4(x)&=-x_2-x_3+2x_4-x_5,\\
\alpha_5(x)&=-x_4+2x_5-x_6, &
\alpha_6(x)&=-x_5+2x_6-x_7,\\
\alpha_7(x)&=-x_6+2x_7-x_8, &
\alpha_8(x)&=-x_7+2x_8.
\end{aligned}
\]

We use \textsc{GAP} to generate the positive roots of
\(\mathfrak E_8(\F_q)\) and then compute their associated linear factors
(see Table \ref{tab:E8-root-linear-forms}).

\begin{tiny}
\begin{longtable}{c c c c}
\captionsetup{font=tiny}
\caption{Positive roots and linear factors for $\mathfrak E_8$.} \\
\hline
Positive root coefficients & Linear factor coefficients & Positive root coefficients & Linear factor coefficients \\
\hline
\endfirsthead
\hline
Positive root coefficients & Linear factor coefficients & Positive root coefficients & Linear factor coefficients \\
\hline
\endhead
$(1,0,0,0,0,0,0,0)$ & $(2,0,-1,0,0,0,0,0)$ & $(0,1,1,2,2,2,1,0)$ & $(-1,0,0,0,0,1,0,-1)$ \\
$(0,1,0,0,0,0,0,0)$ & $(0,2,0,-1,0,0,0,0)$ & $(0,1,1,2,2,1,1,1)$ & $(-1,0,0,0,1,-1,0,1)$ \\
$(0,0,1,0,0,0,0,0)$ & $(-1,0,2,-1,0,0,0,0)$ & $(1,1,2,3,2,1,0,0)$ & $(0,-1,0,1,0,0,-1,0)$ \\
$(0,0,0,1,0,0,0,0)$ & $(0,-1,-1,2,-1,0,0,0)$ & $(1,1,2,2,2,1,1,0)$ & $(0,0,1,-1,1,-1,1,-1)$ \\
$(0,0,0,0,1,0,0,0)$ & $(0,0,0,-1,2,-1,0,0)$ & $(1,1,2,2,1,1,1,1)$ & $(0,0,1,0,-1,0,0,1)$ \\
$(0,0,0,0,0,1,0,0)$ & $(0,0,0,0,-1,2,-1,0)$ & $(1,1,1,2,2,2,1,0)$ & $(1,0,-1,0,0,1,0,-1)$ \\
$(0,0,0,0,0,0,1,0)$ & $(0,0,0,0,0,-1,2,-1)$ & $(1,1,1,2,2,1,1,1)$ & $(1,0,-1,0,1,-1,0,1)$ \\
$(0,0,0,0,0,0,0,1)$ & $(0,0,0,0,0,0,-1,2)$ & $(0,1,1,2,2,2,1,1)$ & $(-1,0,0,0,0,1,-1,1)$ \\
$(1,0,1,0,0,0,0,0)$ & $(1,0,1,-1,0,0,0,0)$ & $(1,2,2,3,2,1,0,0)$ & $(0,1,0,0,0,0,-1,0)$ \\
$(0,1,0,1,0,0,0,0)$ & $(0,1,-1,1,-1,0,0,0)$ & $(1,1,2,3,2,1,1,0)$ & $(0,-1,0,1,0,-1,1,-1)$ \\
$(0,0,1,1,0,0,0,0)$ & $(-1,-1,1,1,-1,0,0,0)$ & $(1,1,2,2,2,2,1,0)$ & $(0,0,1,-1,0,1,0,-1)$ \\
$(0,0,0,1,1,0,0,0)$ & $(0,-1,-1,1,1,-1,0,0)$ & $(1,1,2,2,2,1,1,1)$ & $(0,0,1,-1,1,-1,0,1)$ \\
$(0,0,0,0,1,1,0,0)$ & $(0,0,0,-1,1,1,-1,0)$ & $(1,1,1,2,2,2,1,1)$ & $(1,0,-1,0,0,1,-1,1)$ \\
$(0,0,0,0,0,1,1,0)$ & $(0,0,0,0,-1,1,1,-1)$ & $(0,1,1,2,2,2,2,1)$ & $(-1,0,0,0,0,0,1,0)$ \\
$(0,0,0,0,0,0,1,1)$ & $(0,0,0,0,0,-1,1,1)$ & $(1,2,2,3,2,1,1,0)$ & $(0,1,0,0,0,-1,1,-1)$ \\
$(1,0,1,1,0,0,0,0)$ & $(1,-1,0,1,-1,0,0,0)$ & $(1,1,2,3,2,2,1,0)$ & $(0,-1,0,1,-1,1,0,-1)$ \\
$(0,1,1,1,0,0,0,0)$ & $(-1,1,1,0,-1,0,0,0)$ & $(1,1,2,3,2,1,1,1)$ & $(0,-1,0,1,0,-1,0,1)$ \\
$(0,1,0,1,1,0,0,0)$ & $(0,1,-1,0,1,-1,0,0)$ & $(1,1,2,2,2,2,1,1)$ & $(0,0,1,-1,0,1,-1,1)$ \\
$(0,0,1,1,1,0,0,0)$ & $(-1,-1,1,0,1,-1,0,0)$ & $(1,1,1,2,2,2,2,1)$ & $(1,0,-1,0,0,0,1,0)$ \\
$(0,0,0,1,1,1,0,0)$ & $(0,-1,-1,1,0,1,-1,0)$ & $(1,2,2,3,2,2,1,0)$ & $(0,1,0,0,-1,1,0,-1)$ \\
$(0,0,0,0,1,1,1,0)$ & $(0,0,0,-1,1,0,1,-1)$ & $(1,2,2,3,2,1,1,1)$ & $(0,1,0,0,0,-1,0,1)$ \\
$(0,0,0,0,0,1,1,1)$ & $(0,0,0,0,-1,1,0,1)$ & $(1,1,2,3,3,2,1,0)$ & $(0,-1,0,0,1,0,0,-1)$ \\
$(1,1,1,1,0,0,0,0)$ & $(1,1,0,0,-1,0,0,0)$ & $(1,1,2,3,2,2,1,1)$ & $(0,-1,0,1,-1,1,-1,1)$ \\
$(1,0,1,1,1,0,0,0)$ & $(1,-1,0,0,1,-1,0,0)$ & $(1,1,2,2,2,2,2,1)$ & $(0,0,1,-1,0,0,1,0)$ \\
$(0,1,1,1,1,0,0,0)$ & $(-1,1,1,-1,1,-1,0,0)$ & $(1,2,2,3,3,2,1,0)$ & $(0,1,0,-1,1,0,0,-1)$ \\
$(0,1,0,1,1,1,0,0)$ & $(0,1,-1,0,0,1,-1,0)$ & $(1,2,2,3,2,2,1,1)$ & $(0,1,0,0,-1,1,-1,1)$ \\
$(0,0,1,1,1,1,0,0)$ & $(-1,-1,1,0,0,1,-1,0)$ & $(1,1,2,3,3,2,1,1)$ & $(0,-1,0,0,1,0,-1,1)$ \\
$(0,0,0,1,1,1,1,0)$ & $(0,-1,-1,1,0,0,1,-1)$ & $(1,1,2,3,2,2,2,1)$ & $(0,-1,0,1,-1,0,1,0)$ \\
$(0,0,0,0,1,1,1,1)$ & $(0,0,0,-1,1,0,0,1)$ & $(1,2,2,4,3,2,1,0)$ & $(0,0,-1,1,0,0,0,-1)$ \\
$(1,1,1,1,1,0,0,0)$ & $(1,1,0,-1,1,-1,0,0)$ & $(1,2,2,3,3,2,1,1)$ & $(0,1,0,-1,1,0,-1,1)$ \\
$(1,0,1,1,1,1,0,0)$ & $(1,-1,0,0,0,1,-1,0)$ & $(1,2,2,3,2,2,2,1)$ & $(0,1,0,0,-1,0,1,0)$ \\
$(0,1,1,2,1,0,0,0)$ & $(-1,0,0,1,0,-1,0,0)$ & $(1,1,2,3,3,2,2,1)$ & $(0,-1,0,0,1,-1,1,0)$ \\
$(0,1,1,1,1,1,0,0)$ & $(-1,1,1,-1,0,1,-1,0)$ & $(1,2,3,4,3,2,1,0)$ & $(-1,0,1,0,0,0,0,-1)$ \\
$(0,1,0,1,1,1,1,0)$ & $(0,1,-1,0,0,0,1,-1)$ & $(1,2,2,4,3,2,1,1)$ & $(0,0,-1,1,0,0,-1,1)$ \\
$(0,0,1,1,1,1,1,0)$ & $(-1,-1,1,0,0,0,1,-1)$ & $(1,2,2,3,3,2,2,1)$ & $(0,1,0,-1,1,-1,1,0)$ \\
$(0,0,0,1,1,1,1,1)$ & $(0,-1,-1,1,0,0,0,1)$ & $(1,1,2,3,3,3,2,1)$ & $(0,-1,0,0,0,1,0,0)$ \\
$(1,1,1,2,1,0,0,0)$ & $(1,0,-1,1,0,-1,0,0)$ & $(2,2,3,4,3,2,1,0)$ & $(1,0,0,0,0,0,0,-1)$ \\
$(1,1,1,1,1,1,0,0)$ & $(1,1,0,-1,0,1,-1,0)$ & $(1,2,3,4,3,2,1,1)$ & $(-1,0,1,0,0,0,-1,1)$ \\
$(1,0,1,1,1,1,1,0)$ & $(1,-1,0,0,0,0,1,-1)$ & $(1,2,2,4,3,2,2,1)$ & $(0,0,-1,1,0,-1,1,0)$ \\
$(0,1,1,2,1,1,0,0)$ & $(-1,0,0,1,-1,1,-1,0)$ & $(1,2,2,3,3,3,2,1)$ & $(0,1,0,-1,0,1,0,0)$ \\
$(0,1,1,1,1,1,1,0)$ & $(-1,1,1,-1,0,0,1,-1)$ & $(2,2,3,4,3,2,1,1)$ & $(1,0,0,0,0,0,-1,1)$ \\
$(0,1,0,1,1,1,1,1)$ & $(0,1,-1,0,0,0,0,1)$ & $(1,2,3,4,3,2,2,1)$ & $(-1,0,1,0,0,-1,1,0)$ \\
$(0,0,1,1,1,1,1,1)$ & $(-1,-1,1,0,0,0,0,1)$ & $(1,2,2,4,3,3,2,1)$ & $(0,0,-1,1,-1,1,0,0)$ \\
$(1,1,2,2,1,0,0,0)$ & $(0,0,1,0,0,-1,0,0)$ & $(2,2,3,4,3,2,2,1)$ & $(1,0,0,0,0,-1,1,0)$ \\
$(1,1,1,2,1,1,0,0)$ & $(1,0,-1,1,-1,1,-1,0)$ & $(1,2,3,4,3,3,2,1)$ & $(-1,0,1,0,-1,1,0,0)$ \\
$(1,1,1,1,1,1,1,0)$ & $(1,1,0,-1,0,0,1,-1)$ & $(1,2,2,4,4,3,2,1)$ & $(0,0,-1,0,1,0,0,0)$ \\
$(1,0,1,1,1,1,1,1)$ & $(1,-1,0,0,0,0,0,1)$ & $(2,2,3,4,3,3,2,1)$ & $(1,0,0,0,-1,1,0,0)$ \\
$(0,1,1,2,2,1,0,0)$ & $(-1,0,0,0,1,0,-1,0)$ & $(1,2,3,4,4,3,2,1)$ & $(-1,0,1,-1,1,0,0,0)$ \\
$(0,1,1,2,1,1,1,0)$ & $(-1,0,0,1,-1,0,1,-1)$ & $(2,2,3,4,4,3,2,1)$ & $(1,0,0,-1,1,0,0,0)$ \\
$(0,1,1,1,1,1,1,1)$ & $(-1,1,1,-1,0,0,0,1)$ & $(1,2,3,5,4,3,2,1)$ & $(-1,-1,0,1,0,0,0,0)$ \\
$(1,1,2,2,1,1,0,0)$ & $(0,0,1,0,-1,1,-1,0)$ & $(2,2,3,5,4,3,2,1)$ & $(1,-1,-1,1,0,0,0,0)$ \\
$(1,1,1,2,2,1,0,0)$ & $(1,0,-1,0,1,0,-1,0)$ & $(1,3,3,5,4,3,2,1)$ & $(-1,1,0,0,0,0,0,0)$ \\
$(1,1,1,2,1,1,1,0)$ & $(1,0,-1,1,-1,0,1,-1)$ & $(2,3,3,5,4,3,2,1)$ & $(1,1,-1,0,0,0,0,0)$ \\
$(1,1,1,1,1,1,1,1)$ & $(1,1,0,-1,0,0,0,1)$ & $(2,2,4,5,4,3,2,1)$ & $(0,-1,1,0,0,0,0,0)$ \\
$(0,1,1,2,2,1,1,0)$ & $(-1,0,0,0,1,-1,1,-1)$ & $(2,3,4,5,4,3,2,1)$ & $(0,1,1,-1,0,0,0,0)$ \\
$(0,1,1,2,1,1,1,1)$ & $(-1,0,0,1,-1,0,0,1)$ & $(2,3,4,6,4,3,2,1)$ & $(0,0,0,1,-1,0,0,0)$ \\
$(1,1,2,2,2,1,0,0)$ & $(0,0,1,-1,1,0,-1,0)$ & $(2,3,4,6,5,3,2,1)$ & $(0,0,0,0,1,-1,0,0)$ \\
$(1,1,2,2,1,1,1,0)$ & $(0,0,1,0,-1,0,1,-1)$ & $(2,3,4,6,5,4,2,1)$ & $(0,0,0,0,0,1,-1,0)$ \\
$(1,1,1,2,2,1,1,0)$ & $(1,0,-1,0,1,-1,1,-1)$ & $(2,3,4,6,5,4,3,1)$ & $(0,0,0,0,0,0,1,-1)$ \\
$(1,1,1,2,1,1,1,1)$ & $(1,0,-1,1,-1,0,0,1)$ & $(2,3,4,6,5,4,3,2)$ & $(0,0,0,0,0,0,0,1)$ \\
\hline
\label{tab:E8-root-linear-forms}
\end{longtable}
\end{tiny}

\begin{lemma}\label{lemma (1): consistent for E_8}
The Lie algebra $\mathfrak E_8(\mathbb{F}_q)$ contains a consistent element if $q>3083$ and $\Char (\F_q) \ne 2,3$. 
\end{lemma}

\begin{proof}
The \textsc{GAP} output shows that $C_{\mathfrak E_8}$ has no prime divisor and the polynomial $\hat{f}_{\mathfrak{E}_8}$ contains $4560$ linear factors. Let $d_i$ be the number of linear factors in $\hat{f}_{\mathfrak{E}_8}$ having zero coefficient of $x_1, x_2, \dots, x_{i-1}$ and a non-zero coefficient of $x_i$, then $d_1 = 3083$, $d_2=1049$, $d_3=298$, $d_4=75$, $d_5=34$, $d_6=15$, $d_7=5$ and $d_8=1$.
Also, the coefficient of $x_1^{3083} x_2^{1049} x_3^{298} x_4^{75} x_5^{34} x_6^{15} x_7^5 x_8$ in the product of all linear factors in $\hat{f}_{\mathfrak{E}_8}$ is non-zero if $\Char(\F_q) \ne 2,3$. Hence, by Theorem \ref{Alon's theorem}, if $q>3083$ and $\Char(\F_q) \ne 2,3$, then $\mathfrak E_8(\F_q)$ contains a consistent element. 
\end{proof}

\begin{lemma}\label{lemma: consistent for E_8 (char=2)}
The Lie algebra $\mathfrak E_8(\mathbb{F}_q)$ contains a consistent element if $q>128$ and $\Char (\F_q) = 2$. 
\end{lemma}

\begin{proof}
In this case, the product of scalars obtained by $\textsc {GAP}$ after further normalization is $\equiv 1 \mod{(2)}$. Also, we are left with only $255$ distinct linear factors. The coefficient of $x_1^{128} x_2^{64} x_3^{32} x_4^{16} x_5^8 x_6^4 x_7^2 x_8$ in the product of these distinct linear factors modulo $2$ is non-zero and hence, by Theorem \ref{Alon's theorem}, if $q>128$ and $\Char(\F_q) =2$, then there exists a consistent element in $\mathfrak E_8(\F_q)$. 
\end{proof}

\begin{lemma}\label{lemma: consistent for E_8 (char=3)}
The Lie algebra $\mathfrak E_8(\mathbb{F}_q)$ contains a consistent element if $q>1515$ and $\Char (\F_q) = 3$.
\end{lemma}

\begin{proof}
When $\Char (\F_q) = 3$, further normalization involves scalars whose product is $\equiv 2 \mod{(3)}$ and we have $2320$ distinct linear factors in this case. The coefficient of $x_1^{1515} x_2^{537} x_3^{177} x_4^{56} x_5^{22} x_6^9 x_7^3 x_8$ in the product of these distinct linear factors modulo $3$ is non-zero and hence, by Theorem \ref{Alon's theorem}, there exists a consistent element in this case. 
\end{proof}

\begin{proposition}\label{prop: 2-generation for E_8}
The Lie algebra $\mathfrak E_8(\mathbb{F}_q)$ is $2$-generated.
\end{proposition}

\begin{proof}
By Lemmas~\ref{lemma (1): consistent for E_8}, \ref{lemma: consistent for E_8 (char=2)}, and \ref{lemma: consistent for E_8 (char=3)}, $\mathfrak E_8(\mathbb{F}_q)$ contains a consistent element if $q > 3083$, $\Char(\F_q) \ne 2,3$; $q > 128$, $\Char(\F_q) = 2$; and $q > 1515$, $\Char(\F_q) = 3$. Thus, by Lemma \ref{lem:generation-splitting}, $\mathfrak E_8(\mathbb{F}_q)$ is $2$-generated in this case. For the remaining values of $q$, a similar \textsc{GAP} code as in Listing \ref{lst:gap 2-generation of small q} verifies that $\mathfrak E_8(\mathbb{F}_q)$ is generated by the following elements.

\begin{align*}
 & x := b[1][1] + b[1][2] + b[1][3] + b[1][4]  + b[1][5] + b[1][6] + b[1][7] + b[1][8] + b[3][1],\\
  &y := b[2][1] + b[2][2] + b[2][3] + b[2][4] + b[2][5] + b[2][6] + b[2][7] + b[2][8] + b[3][2].   
\end{align*}
\end{proof}

\begin{proof}[Proof of Theorem \ref{thm:two-generation exceptional}] It follows directly from Propositions \ref{prop: 2-generation for G_2}, \ref{prop: 2-generation for F_4}, \ref{prop: 2-generation for E_6}, \ref{prop: 2-generation for E_7}, and \ref{prop: 2-generation for E_8}.
\end{proof}

\section{Proof of the main theorem}
\label{section:proof}
Finally, we present the proof of our main result, which follows directly from the results in Sections~\ref{section:classical} and~\ref{section:exceptional}.

\begin{proof}[Proof of Theorem~A] The classical case follows from Theorem~\ref{thm:two-generation classical}, the main theorem of \cite{cjz}, and the fact that \(D_1\) and \(D_2\) are not simple. The exceptional case follows from Theorem~\ref{thm:two-generation exceptional} and the fact that \(\G_2(\F_q)\) is not simple when \(\Char(\F_q)=3\). This completes the proof.
\end{proof}

\section{GAP codes}\label{Section: GAP Codes}

\begin{lstlisting}[
   caption={\textsc{GAP} code to show   $\G_2(\F_q)$ is $2$-generated for $q=5,7$ and $11$.},  
   label={lst:gap 2-generation of small q},
   basicstyle=\ttfamily\fontsize{8pt}{9pt}\selectfont
]
LoadPackage("liealgdb");
TestG2Pair := function(q)
    local L, b, x, y, S, dimS;
    
    L := SimpleLieAlgebra("G", 2, GF(q));
    b := ChevalleyBasis(L);    
    x := b[1][1] + b[1][2] + b[3][1];
    y := b[2][1] + b[2][2] + b[3][2];
    S := Subalgebra(L, [x, y]);
    dimS := Dimension(S);
    Print("q = ", q, ", dimension = ", dimS, "\n");
    if dimS = 14 then
        Print("  -> SUCCESS: generates G2 over GF(", q, ")\n");
        return true;
    else
        Print("  -> FAIL: dimension is ", dimS, " (expected 14)\n");
        return false;
    fi;
end;

# Run the test for q = 5, 7, 11
for q in [5, 7, 11] do
    TestG2Pair(q);
    Print("\n");
od;
\end{lstlisting}

\begin{lstlisting}[
    caption={\textsc{GAP} code for \(\mathfrak F_4\) to compute positive roots, distinct linear factors, and other related data.},
    label={lst:gap-for-F4-linear-factors},
    basicstyle=\ttfamily\fontsize{8pt}{9pt}\selectfont
]
CartanF4 := [
    [ 2, -1,  0,  0],
    [-1,  2, -2,  0],
    [ 0, -1,  2, -1],
    [ 0,  0, -1,  2]
];
RootsF4 := IdentityMat(4);
posF4 := 1;

while posF4 <= Length(RootsF4) do
    betaF4 := RootsF4[posF4];
    for iF4 in [1..4] do
        valF4 := Sum([1..4],
            jF4 -> betaF4[jF4] * CartanF4[jF4][iF4]);
        newF4 := ShallowCopy(betaF4);
        newF4[iF4] := newF4[iF4] - valF4;

        if Position(RootsF4, newF4) = fail then
            Add(RootsF4, newF4);
        fi;
    od;
    posF4 := posF4 + 1;
od;

PosRootsF4 := Filtered(RootsF4,
    betaF4 -> ForAll(betaF4, cF4 -> cF4 >= 0));

# Root coefficients to linear-form coefficients: d = c A
RootFormF4 := function(betaF4)
    return List([1..4], iF4 ->
        Sum([1..4],
            jF4 -> betaF4[jF4] * CartanF4[jF4][iF4]));
end;

LinearFormsF4 := List(PosRootsF4, RootFormF4);
Print("Total roots: ", Length(RootsF4), "\n");
Print("Positive roots: ", Length(PosRootsF4), "\n");

if Length(RootsF4) <> 48 or Length(PosRootsF4) <> 24 then
    Error("Expected 48 roots and 24 positive roots. Check CartanF4.");
fi;
Print("\n--- Positive roots and corresponding linear forms ---\n");

for iF4 in [1..Length(PosRootsF4)] do
    Print("Root coefficients: ", PosRootsF4[iF4],
          "   Linear form coefficients: ",
          LinearFormsF4[iF4], "\n");
od;

all_factors := [];
# 1. Linear forms themselves
for r in LinearFormsF4 do
    Add(all_factors, r);
od;
# 2. Differences r_i - r_j, i < j
for i in [1..Length(LinearFormsF4)] do
    for j in [i+1..Length(LinearFormsF4)] do
        Add(all_factors, LinearFormsF4[i] - LinearFormsF4[j]);
    od;
od;
# 3. Sums r_i + r_j, ordered pairs i <> j
for i in [1..Length(LinearFormsF4)] do
    for j in [1..Length(LinearFormsF4)] do
        if i <> j then
            Add(all_factors, LinearFormsF4[i] + LinearFormsF4[j]);
        fi;
    od;
od;
Print("\nTotal factors (with repetitions): ",
      Length(all_factors), "\n");

# Normalize a 4-vector over Q
NormalizeForm := function(v)
    local w, g, denoms, commonDenom, i;

    w := ShallowCopy(v);
    denoms := List(w, x -> DenominatorRat(x));
    commonDenom := Lcm(denoms);
    w := List(w, x -> NumeratorRat(x * commonDenom));
    g := Gcd(List(w, x -> AbsInt(x)));

    if g = 0 then
        return [0,0,0,0];
    fi;
    w := List(w, x -> QuoInt(x, g));
    for i in [1..Length(w)] do
        if w[i] <> 0 then
            if w[i] < 0 then
                w := List(w, x -> -x);
            fi;
            break;
        fi;
    od;
    return w;
end;

# Deduplicate primitive forms and retain their occurrences
distinct_forms := [];
scalars_per_form := [];
for f in all_factors do
    nf := NormalizeForm(f);
    pos := Position(distinct_forms, nf);

    if pos = fail then
        Add(distinct_forms, nf);
        Add(scalars_per_form, [f]);
    else
        Add(scalars_per_form[pos], f);
    fi;
od;
Print("\nNumber of distinct primitive forms: ",
      Length(distinct_forms), "\n");
Print("Sum of multiplicities: ",
      Sum(List(scalars_per_form, Length)), "\n");

# Product of all scalars relating occurrences to normalized forms
prodScalars := 1;
for i in [1..Length(distinct_forms)] do
    nf := distinct_forms[i];
    for f in scalars_per_form[i] do
        for k in [1..4] do
            if nf[k] <> 0 then
                prodScalars := prodScalars * (f[k] / nf[k]);
                break;
            fi;
        od;
    od;
od;
Print("\nProduct of all scalars = ", prodScalars, "\n");

PrimeList := function(n)
    n := AbsInt(n);
    if n <= 1 then
        return [];
    fi;
    return Set(FactorsInt(n));
end;

PrintFactorization := function(label, n)
    local fs;
    n := AbsInt(n);
    Print(label, " = ", n, "\n");
    if n = 1 then
        Print("  (trivial: no prime factors)\n");
        return;
    fi;
    fs := FactorsInt(n);
    Print("  Factorization (with multiplicity): ", fs, "\n");
    Print("  Distinct primes: ", Set(fs), "\n");
end;

# Factorization and prime divisors of prodScalars
Print("\n--- Factorization of prodScalars ---\n");
S := AbsInt(prodScalars);
Print("prodScalars = ", prodScalars, "\n");
Print("Using S = |prodScalars| = ", S, "\n");
PrintFactorization("|prodScalars|", S);
primeDivisorsS := PrimeList(S);
Print("Prime divisors of S: ", primeDivisorsS, "\n");

# Count distinct forms by first nonzero coordinate over Q
Print("\n--- Counting distinct forms by zero patterns (over Q) ---\n");
counts := [0,0,0,0];
for f in distinct_forms do
    for i in [1..4] do
        if f[i] <> 0 then
            counts[i] := counts[i] + 1;
            break;
        fi;
    od;
od;
Print("Total distinct forms: ", Length(distinct_forms), "\n");
for i in [1..4] do
    Print("Forms whose first nonzero coefficient is x",
          i, ": ", counts[i], "\n");
od;
n_zero := Number(distinct_forms,
    f -> ForAll(f, x -> x = 0));
Print("Zero forms: ", n_zero, "\n");

# Product of first nonzero coefficients over Q
ProductFirstNonZero := function(forms)
    local prod, f, i, found;
    prod := 1;
    for f in forms do
        found := false;
        for i in [1..Length(f)] do
            if f[i] <> 0 then
                prod := prod * f[i];
                found := true;
                break;
            fi;
        od;
        if not found then
            Print("Warning: zero form encountered!\n");
        fi;
    od;
    return prod;
end;
P_first_nonzero := ProductFirstNonZero(distinct_forms);
Print("\nProduct of first non-zero coefficients = ",
      P_first_nonzero, "\n");
Print("Factorization of |P_first_nonzero| = ",
      FactorsInt(AbsInt(P_first_nonzero)), "\n");
Print("Distinct primes of |P_first_nonzero| = ",
      PrimeList(P_first_nonzero), "\n");

# Per-prime GF(p) calculations for
# p in PrimeList(P_first_nonzero) \ PrimeList(S)
Print("Per-prime GF(p) calculation (with projective deduplication)\n");

scalar_primes := PrimeList(S);
P_primes := PrimeList(P_first_nonzero);
target_primes := Difference(P_primes, scalar_primes);

Print("Primes dividing S: ", scalar_primes, "\n");
Print("Primes dividing P_first_nonzero: ", P_primes, "\n");
Print("Target primes: ", target_primes, "\n");

ProductFirstNonZeroModP := function(forms, p)
    local F, prod, f, i, found;
    F := GF(p);
    prod := One(F);
    for f in forms do
        found := false;
        for i in [1..Length(f)] do
            if f[i] <> Zero(F) then
                prod := prod * f[i];
                found := true;
                break;
            fi;
        od;
        if not found then
            Print("Warning: zero form encountered mod ", p, "\n");
        fi;
    od;
    return prod;
end;

# Projective normalization in GF(p)
NormalizeModP := function(v, F)
    local i, inv;
    for i in [1..Length(v)] do
        if v[i] <> Zero(F) then
            inv := Inverse(v[i]);
            return List(v, x -> x * inv);
        fi;
    od;
    return ShallowCopy(v);
end;

# Per-prime calculations
for p in target_primes do
    Print("Characteristic p = ", p, "\n");
    Print("-------------------------------------------------\n");

    F := GF(p);

    # 1. Reduce distinct integer forms modulo p
    reduced := List(distinct_forms,
        f -> List(f, x -> Int(x) * One(F)));

    # 2. Projective normalization
    normalized := List(reduced, v -> NormalizeModP(v, F));

    # 3. Group by projective class and record scalar lambda.
    # reduced[i] = lambda * representative; zero has lambda = 0.

    classes := [];
    lambdas := [];
    for i in [1..Length(reduced)] do
        v := reduced[i];
        if ForAll(v, x -> x = Zero(F)) then
            pos := Position(classes, v);
            if pos = fail then
                Add(classes, v);
                pos := Length(classes);
            fi;
            Add(lambdas, Zero(F));
         else
            nv := normalized[i];
            pos := Position(classes, nv);
            if pos = fail then
                Add(classes, nv);
                pos := Length(classes);
            fi;
            for k in [1..4] do
                if v[k] <> Zero(F) then
                    Add(lambdas, v[k] / classes[pos][k]);
                    break;
                fi;
            od;
        fi;
    od;
    distinct_forms_modp := ShallowCopy(classes);

    Print("\n--- Projective deduplication mod ", p, " ---\n");
    Print("Forms before deduplication: ", Length(reduced), "\n");
    Print("Distinct projective classes mod ", p, ": ",
          Length(distinct_forms_modp), "\n");

    n_new := Number(distinct_forms_modp,
        f -> not ForAll(f, x -> x = Zero(F)));
    if n_new = 0 then
        Print("Every form collapses to zero mod ", p, ".\n");
    fi;

    # 4. Product of collapse scalars modulo p
    prodScalarsModP := One(F);
    for lam in lambdas do
        prodScalarsModP := prodScalarsModP * lam;
    od;
    Print("\nProduct of collapse scalars mod ", p, " = ",
          Int(prodScalarsModP), "\n");
    if Int(prodScalarsModP) > 1 then
        Print("Factorization: ",
              FactorsInt(Int(prodScalarsModP)), "\n");
        Print("Distinct primes: ",
              PrimeList(Int(prodScalarsModP)), "\n");
    else
        Print("Product is 0 or 1; no nontrivial prime factorization.\n");
    fi;

    # 5. Zero-pattern counts on the deduplicated set
    Print("\n--- Zero-pattern counts mod ", p, " ---\n");
    counts_p := [0,0,0,0];
    for f in distinct_forms_modp do
        for i in [1..4] do
            if f[i] <> Zero(F) then
                counts_p[i] := counts_p[i] + 1;
                break;
            fi;
        od;
    od;
    Print("Total distinct forms mod ", p, ": ",
          Length(distinct_forms_modp), "\n");

    for i in [1..4] do
        Print("Forms whose first nonzero coefficient is x",
              i, " mod ", p, ": ", counts_p[i], "\n");
    od;

    n_zero := Number(distinct_forms_modp,
        f -> ForAll(f, x -> x = Zero(F)));
    Print("Forms reducing to the zero vector mod ", p, ": ",
          n_zero, "\n");

# 6. Product of first nonzero coefficients modulo p
Print("\n--- Product of first nonzero coefficients mod ", p, " ---\n");
P_p := ProductFirstNonZeroModP(distinct_forms_modp, p);
Print("Product modulo ", p, " = ", P_p, "\n");
od;
Print("Done. Processed target primes: ", target_primes, "\n");
\end{lstlisting}


\bibliographystyle{amsalpha}
\bibliography{Reference}
\end{document}